\documentclass[12pt]{article} \usepackage{amsthm}     

\usepackage{enumerate}
\usepackage{amssymb}
\usepackage{parskip}
\usepackage{subfigure}
\usepackage{epsfig}
\usepackage{graphicx}
\usepackage{amsmath}
\usepackage{cases}
\usepackage{mathtools}
\usepackage{appendix}
\usepackage{pgf,tikz}
\usetikzlibrary{arrows}
\usetikzlibrary{decorations.pathmorphing}

\DeclareMathOperator{\sign}{sgn}
\DeclareMathOperator{\erfc}{erfc}
\begin{document}
\title
{Shadow boundary effects in hybrid numerical-asymptotic methods for high frequency scattering}
\author{D.~P.~Hewett\\
\footnotesize{Department of Mathematics and Statistics, University of Reading, UK}\\
\footnotesize{(Current address: Mathematical Institute, University of Oxford, UK)}\\
\footnotesize{Email: hewett@maths.ox.ac.uk}
}
\maketitle
\newcommand{\done}[2]{\dfrac{d {#1}}{d {#2}}}
\newcommand{\donet}[2]{\frac{d {#1}}{d {#2}}}
\newcommand{\pdone}[2]{\dfrac{\partial {#1}}{\partial {#2}}}
\newcommand{\pdonet}[2]{\frac{\partial {#1}}{\partial {#2}}}
\newcommand{\pdonetext}[2]{\partial {#1}/\partial {#2}}
\newcommand{\pdtwo}[2]{\dfrac{\partial^2 {#1}}{\partial {#2}^2}}
\newcommand{\pdtwot}[2]{\frac{\partial^2 {#1}}{\partial {#2}^2}}
\newcommand{\pdtwomix}[3]{\dfrac{\partial^2 {#1}}{\partial {#2}\partial {#3}}}
\newcommand{\pdtwomixt}[3]{\frac{\partial^2 {#1}}{\partial {#2}\partial {#3}}}
\newcommand{\bs}[1]{\mathbf{#1}}
\newcommand{\bx}{\mathbf{x}}
\newcommand{\by}{\mathbf{y}}
\newcommand{\bd}{\mathbf{d}} 
\newcommand{\bn}{\mathbf{n}} 
\newcommand{\bP}{\mathbf{P}} 
\newcommand{\bp}{\mathbf{p}} 
\newcommand{\ol}[1]{\overline{#1}}
\newcommand{\rf}[1]{(\ref{#1})}
\newcommand{\xt}{\mathbf{x},t}
\newcommand{\hs}[1]{\hspace{#1mm}}
\newcommand{\vs}[1]{\vspace{#1mm}}
\newcommand{\eps}{\varepsilon}
\newcommand{\ord}[1]{\mathcal{O}\left(#1\right)} 
\newcommand{\oord}[1]{o\left(#1\right)}
\newcommand{\Ord}[1]{\Theta\left(#1\right)}
\newcommand{\PhiF}{\Phi_{\rm freq}}
\newcommand{\real}[1]{{\rm Re}\left[#1\right]} 
\newcommand{\im}[1]{{\rm Im}\left[#1\right]}
\newcommand{\hsnorm}[1]{||#1||_{H^{s}(\bs{R})}}
\newcommand{\hnorm}[1]{||#1||_{\tilde{H}^{-1/2}((0,1))}}
\newcommand{\norm}[2]{\left\|#1\right\|_{#2}}
\newcommand{\normt}[2]{\|#1\|_{#2}}
\newcommand{\on}[1]{\Vert{#1} \Vert_{1}}
\newcommand{\tn}[1]{\Vert{#1} \Vert_{2}}
\newcommand{\ts}{\tilde{s}}
\newcommand{\tGamma}{{\tilde{\Gamma}}}
\newcommand{\darg}[1]{\left|{\rm arg}\left[ #1 \right]\right|}
\newcommand{\bnabla}{\boldsymbol{\nabla}}
\newcommand{\dive}{\boldsymbol{\nabla}\cdot}
\newcommand{\curl}{\boldsymbol{\nabla}\times}
\newcommand{\Phixy}{\Phi(\bx,\by)}
\newcommand{\PhiOxy}{\Phi_0(\bx,\by)}
\newcommand{\dxPhixy}{\pdone{\Phi}{n(\bx)}(\bx,\by)}
\newcommand{\dyPhixy}{\pdone{\Phi}{n(\by)}(\bx,\by)}
\newcommand{\dxPhiOxy}{\pdone{\Phi_0}{n(\bx)}(\bx,\by)}
\newcommand{\dyPhiOxy}{\pdone{\Phi_0}{n(\by)}(\bx,\by)}

\newcommand{\rd}{\mathrm{d}}
\newcommand{\R}{\mathbb{R}}
\newcommand{\N}{\mathbb{N}}
\newcommand{\Z}{\mathbb{Z}}
\newcommand{\C}{\mathbb{C}}
\newcommand{\K}{{\mathbb{K}}}
\newcommand{\ri}{{\mathrm{i}}}
\newcommand{\re}{{\mathrm{e}}} 

\newcommand{\cA}{\mathcal{A}}
\newcommand{\cC}{\mathcal{C}}
\newcommand{\cS}{\mathcal{S}}
\newcommand{\cD}{\mathcal{D}}
\newcommand{\cone}{{c_{j}^\pm}}
\newcommand{\ctwo}{{c_{2,j}^\pm}}
\newcommand{\cthree}{{c_{3,j}^\pm}}

\newtheorem{thm}{Theorem}[section]
\newtheorem{lem}[thm]{Lemma}
\newtheorem{defn}[thm]{Definition}
\newtheorem{prop}[thm]{Proposition}
\newtheorem{cor}[thm]{Corollary}
\newtheorem{rem}[thm]{Remark}
\newtheorem{conj}[thm]{Conjecture}
\newtheorem{ass}[thm]{Assumption}
\definecolor{cqcqcq}{rgb}{0.9,0.9,0.9}
\newcommand{\cR}{\mathcal{R}}
\newcommand{\bQ}{\mathbf{Q}}
\newcommand{\bR}{\mathbf{R}}
\newcommand{\bS}{\mathbf{S}}
\newcommand{\xSB}{\mathbf{x}_{\rm SB}}
\newcommand{\tx}{\tilde{x}}
\newcommand{\tbx}{\tilde{\mathbf{x}}}
\newcommand{\Fr}{{\rm Fr}}
\newcommand{\dhc}[1]{{\color{red}{#1}}}
\begin{abstract}
The hybrid numerical-asymptotic (HNA) approach aims to reduce the computational cost of conventional numerical methods for high frequency wave scattering problems by enriching the numerical approximation space with oscillatory basis functions, chosen based on partial knowledge of the high frequency solution asymptotics. In this paper we propose a new methodology for the treatment of shadow boundary effects in HNA boundary element methods, using the classical geometrical theory of diffraction phase functions combined with mesh refinement. We develop our methodology in the context of scattering by a class of sound-soft nonconvex polygons, presenting a rigorous numerical analysis (supported by numerical results) which proves the effectiveness of our HNA approximation space at high frequencies. Our analysis is based on a study of certain approximation properties of the Fresnel integral and related functions, which govern the shadow boundary behaviour. 
\end{abstract}
%
%
\section{Introduction}
\label{sec:Intro}
The efficient numerical solution of wave scattering problems for the Helmholtz equation
\begin{align}
\label{eqn:HE}
(\Delta + k^2)u = 0, \qquad k>0,
\end{align}
is important in many areas of science and engineering. 
Conventional finite element method (FEM) or boundary element method (BEM) approaches with piecewise polynomial approximation spaces suffer from the limitation that a fixed number of degrees of freedom $M$ are required per wavelength in order to accurately represent the oscillations in the scattered wave, with $M=10$ being the accepted guideline in the engineering literature (see, e.g.~\cite{PlaneWave} and the references therein). This means that if $L$ is a linear dimension of the (bounded) scattering object, $\lambda$ is the wavelength and $k=\tfrac{2\pi}{\lambda}$ is the wavenumber (proportional to frequency), then the total number of degrees of freedom required is at least proportional to $(MkL)^d$ for FEM and $(MkL)^{d-1}$ for BEM, where $d=2$ or $3$ is the number of space dimensions of the problem. As a result, when $kL$ is large (as is the case in many applications), these conventional approaches are computationally expensive.

Recent research has led to the development of a number of novel FEM/BEM approaches which aim to reduce the number of degrees of freedom required when $kL$ is large by enriching the conventional piecewise polynomial approximation spaces with oscillatory basis functions (see e.g.\ \cite{PlaneWave,ChGrLaSp:11} and the many references therein). 
Our focus in this article is on the so-called hybrid numerical-asymptotic (HNA) BEM approach (recently reviewed in \cite{ChGrLaSp:11}), in which oscillatory BEM basis functions are chosen using partial knowledge of the high frequency solution asymptotics. These asymptotics can be understood within the context of Keller's celebrated geometrical theory of diffraction (GTD) (see e.g.\ \cite{KellerGTD,Borovikov}), in which the wave field is expressed as a sum of leading order geometrical optics (GO) components (the incident and reflected fields) and higher order diffracted components. 
In the HNA approach, each of the GTD components is represented in the BEM approximation space by an appropriate oscillatory function (chosen \emph{a priori}) multiplied by a piecewise polynomial amplitude (to be determined by the BEM) - for details see \cite{ChGrLaSp:11}. 

Since the nature and complexity of the HNA approximation space is linked to that of the underlying high frequency asymptotics, the HNA approach has been applied so far only to a limited number of problems for which these asymptotics are relatively simple (mostly 2D problems, with the exception of \cite{GH11} and \cite[\S7.6]{ChGrLaSp:11}, and mostly convex scatterers, with the exception of \cite{NonConvex}). But for many such problems (e.g.,\ scattering by sound-soft smooth convex obstacles in 2D \cite{DGS07,AsHu:10}, convex \cite{Convex,HeLaMe:11} and nonconvex \cite{NonConvex} polygons and 2D planar screens \cite{ScreenBEM} - see \cite{ChGrLaSp:11} for further examples) the HNA approach has proved to be very effective, providing a dramatic reduction in the number of degrees of freedom required when $kL$ is large, and in some cases even frequency-independent computational cost (when the numerical integration required for practical implementation is carried out using appropriate oscillatory integration routines), see e.g.\ \cite{ScreenBEM}. 


One of the key difficulties one encounters when attempting to apply the HNA methodology to more complex scattering problems involving nonconvex and/or 3D scatterers, is the need to deal with the complicated solution behaviour that occurs near the \emph{shadow boundaries} across which GTD components switch on/off (for an example see Figure~\ref{fig:total1}). Near such shadow boundaries the classical GTD approximation breaks down: on a shadow boundary the phase of the GTD component being switched on/off coincides with that of another higher-order GTD component, and the diffraction coefficient of this higher-order component blows up to infinity. The full wave solution varies smoothly (but rapidly) across such shadow boundaries, but to capture this rapid variation in an asymptotic approximation one has to employ more complicated uniform approximations involving the exact solutions of appropriate canonical diffraction problems which capture the shadow boundary behaviour in question, if such solutions are available (see e.g.\ \cite{UTD,Borovikov,TeChKiOcSmZa:00,FLP}). 


\begin{figure}[t]
\begin{center}
\subfigure[$\alpha=\tfrac{5\pi}{6}$ \label{fig:IncShad}]{
\begin{tikzpicture}[>=triangle 45,x=1.0cm,y=1.0cm, scale=1]
\node (0,0) {\includegraphics[width=6cm]{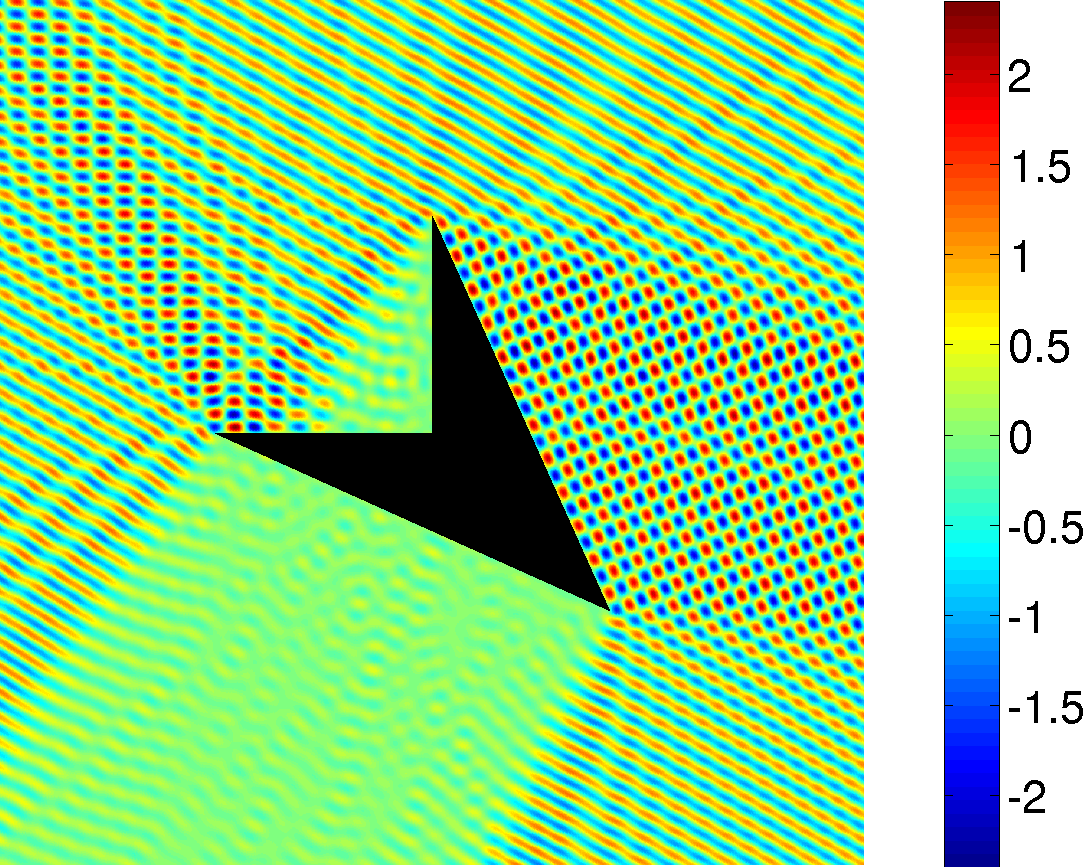}};
\draw[->,thick] (0.05,2.3) -- ++(-0.7*0.707,-0.7*1.225);
\draw (0.1,1.8) node {$\mathbf{d}$};
\draw[dashed,thick] (-0.6,1.14)--++(-1*0.707,-1*1.225);
\end{tikzpicture}
}
\subfigure[$\alpha=\tfrac{7\pi}{6}$ \label{fig:RefShad}]{
\begin{tikzpicture}[>=triangle 45,x=1.0cm,y=1.0cm, scale=1]
\node (0,0) {\includegraphics[width=6cm]{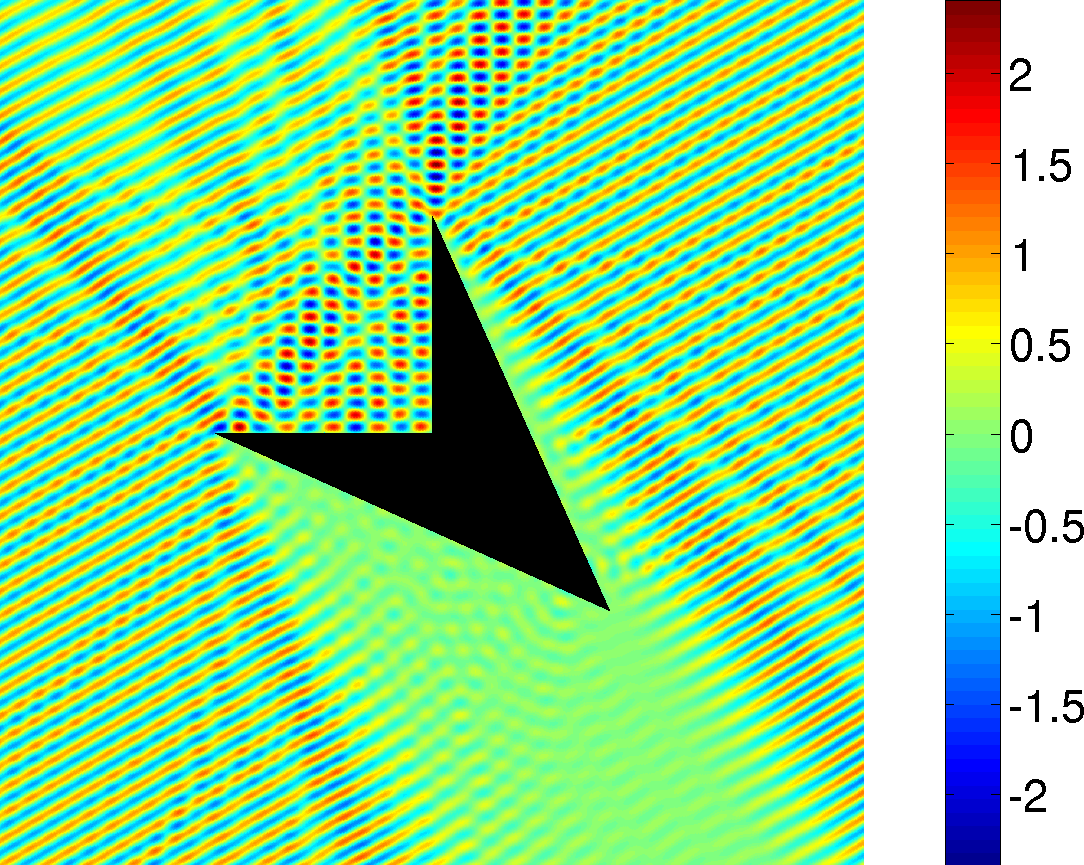}};
\draw[->,thick] (-1.25,2.3) -- ++(0.7*0.707,-0.7*1.225);
\draw (-1.3,1.8) node {$\mathbf{d}$};
\draw[dashed,thick] (-0.6,1.14)--++(-1*0.707,-1*1.225);
\end{tikzpicture}
}\end{center}
\caption{Real part of the total field for scattering of a plane wave by a sound-soft nonconvex polygon of the form illustrated in Figure~\ref{BEMmodelproblem}, for two different incidence angles. In both cases a shadow boundary (shown as a dotted line) emanating from the top vertex of the polygon intersects the horizontal nonconvex side $\Gamma_{\rm nc}$: in (a) this is the shadow boundary associated with the incident wave; in (b) this is the shadow boundary associated with the reflection of the incident wave in the side $\Gamma_{\rm nc}'$.
Here the convex sides have equal length $L_{\rm c}= L_{\rm c}'= 4\pi$ and the nonconvex sides have equal length $L_{\rm nc}= L_{\rm nc}'= 2\pi$. In $(x_1,x_2)$ coordinates the vertices lie at $\bP = (-2\pi,-2\pi)$, $\bQ=(0,-2\pi)$, $\bR=(0,0)$, $\bS=((1+\sqrt{7})\pi,(1-\sqrt{7})\pi)$. The wavenumber $k=10$.
\label{fig:total1}}
\end{figure}
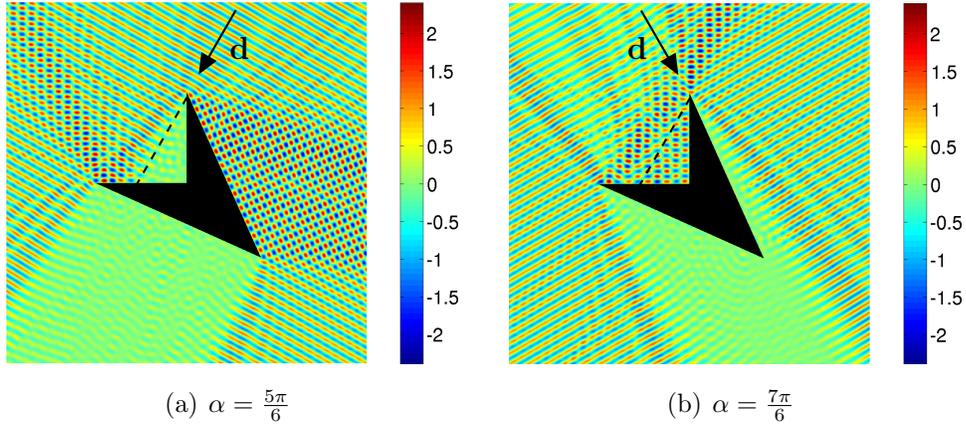

This suggests a natural way to deal with shadow boundaries in the context of HNA methods, namely to mirror the modification of the GTD approximation described above, and include in the HNA approximation space the appropriate (oscillatory) canonical solutions. 
This approach has already been implemented in 2D in the context of scattering by a class of nonconvex polygons in \cite{NonConvex}. However, while the apparent simplicity of this approach is appealing, its applicability is limited because suitable canonical solutions are available for only a few types of shadow boundary, for example that arising in the diffraction of a plane wave by a sound-hard or sound-soft wedge (the latter case being applied in \cite{NonConvex}). For many types of shadow boundary a convenient exact solution to the relevant canonical problem is not available, for example the diffraction of a plane wave by a penetrable (transmission) wedge (which would be relevant to the study of scattering by penetrable polygons - see e.g.\ \cite{GrHeLa:13}), or the diffraction of a plane wave by a sound-soft or sound-hard quarter plane (which would be relevant to the study of scattering by rectangular screens in 3D, see e.g.\ \cite[\S7.6]{ChGrLaSp:11}). 

In this paper we propose a new, more general methodology for the treatment of shadow boundaries in HNA methods, based on local mesh refinement. 
Our proposed methodology uses an HNA approximation space built from the classical GTD components, and does not rely on the existence of canonical solutions. Instead, at a shadow boundary we propose to cut off sharply (with a jump discontinuity) the GTD component that is being switched on/off. For the associated higher-order GTD component (whose classical GTD diffraction coefficient blows up at the shadow boundary), we propose that the mesh associated with its piecewise polynomial approximation in the HNA method should be appropriately adapted so as to (i) accurately capture the expected rapid variation in its amplitude near the shadow boundary, and (ii) compensate for the jump discontinuity artificially introduced when sharply cutting on/off the other GTD component. Specifically, we propose that the intersection $\Lambda_{\rm SB}$ of the shadow boundary with the scatterer boundary $\Gamma$ ($\Lambda_{\rm SB}$ will be a single point on $\Gamma$ in 2D and a curve on $\Gamma$ in 3D) should form part of the mesh skeleton, and that the mesh should be refined towards $\Lambda_{\rm SB}$.


To describe our methodology in more detail we shall focus on the specific 2D problem of scattering of a plane wave by a sound-soft nonconvex polygon of the type considered in \cite{NonConvex}. 
In this case the shadow boundary behaviour is governed by a canonical solution involving the Fresnel integral; the main results of this paper therefore concern regularity and approximation properties of the Fresnel integral and related functions. 
For ease of exposition, we shall restrict our attention to the simple case of a quadrilateral with a right-angled nonconvexity, as illustrated in Figure~\ref{BEMmodelproblem}. But the results presented below can also be applied, with the appropriate modifications, to all of the polygons in the class defined in \cite[Definition 3.1]{NonConvex}. 
As such, this paper represents a proof of concept that mesh refinement can be used to deal with shadow boundary behaviour in HNA methods. 
We expect that the general philosophy of our approach (as described in the previous paragraph) should apply generically (in 3D problems as well as 2D ones), but the details of the appropriate mesh design could vary considerably depending on the nature of the shadow boundaries under consideration, and we leave further investigation of this to future work.


The structure of the paper is as follows. In \S\ref{sec:polygons} we state the scattering problem to be solved, and in \S\ref{sec:BIEFormulation} recall its boundary integral equation (BIE) reformulation. In \S\ref{sec:HFBehaviour} we consider the high frequency asymptotic behaviour of the solution of the BIE, our main new result being Theorem~\ref{thm:PsiDecomp}, which provides a representation of the solution near shadow boundaries which is suitable for HNA approximation, along with the regularity results required for rigorous numerical analysis. This analysis is carried out in \S\ref{sec:HNAApproxSpace}, where we prove in Theorem \ref{thm:BestApprox} that our new approximation space, treating the shadow boundary using mesh refinement, achieves the same qualititative performance (in terms of the number of degrees of freedom required at high frequencies) as the scheme proposed in \cite{NonConvex}, without the need to evaluate any canonical diffraction solutions. In \S\ref{sec:EApprox} we prove a number of approximation properties of the Fresnel integral and related functions, which underpin the analysis in \S\ref{sec:polygons}. Given the ubiquity of the Fresnel integral in the mathematical description of shadow boundary phenomena (see, e.g., \cite{TeChKiOcSmZa:00,OckTew:12,FLP}), these results may be of some independent interest beyond the scattering problem considered in this paper. Finally, in \S\ref{sec:NumResults} we present some numerical results to validate our theoretical error estimates.

\section{Scattering by a sound-soft polygon}
\label{sec:polygons}

\begin{figure}[t]
\begin{center}
 \subfigure[The scatterer $\Omega$ and the local coordinates $(x_1,x_2)$, $(r,\theta)$ and $s$ on 
$\Gamma_{\rm nc}$.
\label{fig:Geometry}]{
	\begin{tikzpicture}[line cap=round,line join=round,>=triangle 45,x=1.0cm,y=1.0cm, scale=0.8]
	\fill[line width=0pt,color=cqcqcq,fill=cqcqcq] (7.47,-2.47)--(5,3)--  (5,0)-- (2,0)-- cycle;
		\draw [line width=1.2pt] (7.47,-2.47)--(5,3)--  (5,0)-- (2,0)-- (7.47,-2.47);
	\draw (4.6,3.2) node {$\bR$};
	\draw (5.3,-0.2) node {$\bQ$};
	\draw (2,0.3) node {$\bP$};
		\draw (7.8,-2.5) node {$\bS$};
	\draw (4.5,1.05) node {$\Gamma_{\rm nc}'$};
	\draw (3.9,0.35) node{$\Gamma_{\rm nc}$};
		\draw (6.8,0.2) node{$\Gamma_{\rm c}$};
			\draw (4.5,-1.7) node{$\Gamma_{\rm c}'$};
	\draw (5,0.38)-- (4.58,0.38);
	\draw (4.58,0.38)-- (4.58,0);
	\draw [->] (5,3) -- (5,5);
	\draw [->] (5,3) -- (7,3);
	\draw (7.1,3.3) node {$x_1$};
	\draw (5.5,4.8) node {$x_2$};
	\draw [->] (5,2.2) arc (-90:26.57:0.8);
	\draw (6,2.7) node {$\alpha$};
		\draw [dashed] (5,3) -- (8,4.5);
	\begin{scope}[xshift=-0.5cm, yshift=-0.25cm]
	\draw (9.5,4) -- (8.5,6) ;
	\draw (9.3,3.9) -- (8.3,5.9) ;
	\draw (9.1,3.8) -- (8.1,5.8) ;
	\draw [thick,->] (9,5) -- (8,4.5) ;	
	\draw (9.5,5.2) node {$u^i$};
	\draw (8.5,4.2) node {$\bd^i$};
		\end{scope}
	\draw (6,-1.1) node {$\Omega$};
	\filldraw (2.9,0) circle (3pt);
	\draw (2.8,0.3) node {$\bx$};
	\draw [<->] (3.05,0.2) -- (4.9,2.8);
	\draw (3.7,1.5) node {$r$};
	\draw [->] (5,1.6) arc (-90:237:1.4);
	\draw (4,4.4) node {$\theta$};
	\draw [<->] (3.0,-0.2) -- (4.9,-0.2);
	\draw (3.9,-0.5) node {$s$};
	\draw (7.4,1.8) node {$D$};
	\draw [thick,->] (7,-1.4) -- ++ (0.45*2.21,0.45) ;	
	\draw (7.6,-1.5) node {$\bn$};
\end{tikzpicture}
}
\hs{5}
 \subfigure[Diffraction by the wedge (shaded) formed by extending $\Gamma_{nc}'$ and $\Gamma_{c}$ to infinity, or the knife edge (thick line) formed by extending $\Gamma_{nc}'$ to infinity.
 \label{fig:KnifeEdge}]{
	\begin{tikzpicture}[line cap=round,line join=round,>=triangle 45,x=1.0cm,y=1.0cm, scale=0.8]
	\fill[line width=0pt,color=cqcqcq,fill=cqcqcq] (7.47,-2.47)--(5,3)--  (5,-2.47)-- cycle;
	\draw [line width=1.8pt] (5,3)-- (5,-2.47);
	\draw [line width=0.8pt] (5,3)-- (7.47,-2.47);
	\draw [->] (5,2.2) arc (-90:26.57:0.8);
	\draw (6,2.6) node {$\alpha$};
		\draw [dashed] (5,3) -- (8,4.5);
	\begin{scope}[xshift=-0.5cm, yshift=-0.25cm]
	\draw (9.5,4) -- (8.5,6) ;
	\draw (9.3,3.9) -- (8.3,5.9) ;
	\draw (9.1,3.8) -- (8.1,5.8) ;
	\draw [thick,->] (9,5) -- (8,4.5) ;	
	\draw (9.5,5.2) node {$u^i$};
	\draw (8.5,4.2) node {$\bd^i$};
		\end{scope}
	\draw (4.6,3.2) node {$\bR$};
	\filldraw (2.9,0) circle (3pt);
	\draw (2.8,0.3) node {$\bx$};
	\draw [<->] (3.05,0.2) -- (4.9,2.8);
	\draw (3.7,1.5) node {$r$};
	\draw [->] (5,1.6) arc (-90:237:1.4);
	\draw (4,4.4) node {$\theta$};
\end{tikzpicture}
}
\end{center}
\caption{Geometry of the scattering problem.
\label{BEMmodelproblem}}
\end{figure}
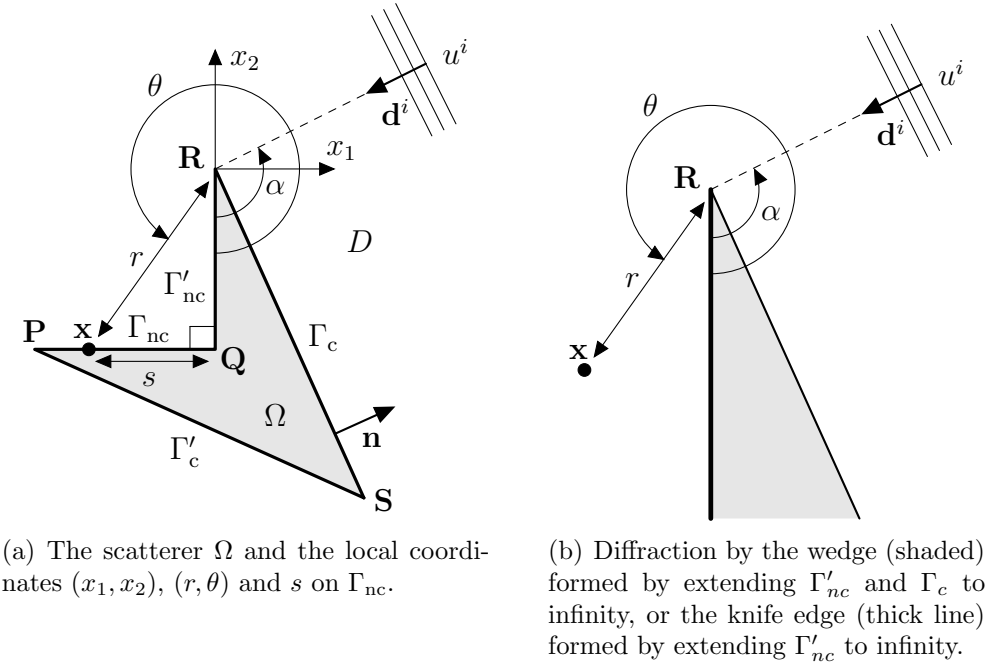

We consider the scattering of a time harmonic incident plane wave
\begin{align}
u^i(\bx)&:=\re^{\ri k \bx\cdot \bd^i},
\qquad \bx=(x_1,x_2)\in\mathbb{R}^2,
\qquad \bd^i=(d^i_1,d^i_2)\in\mathbb{R}^2,\,\,|\bd^i|=1,
\label{eqn:1}
\end{align}
by a sound-soft polygon $\Omega$ of the form illustrated in Figure~\ref{fig:Geometry}. 
We seek a total field $u$ satisfying \rf{eqn:HE} in the exterior domain $D:=\mathbb{R}^2\backslash\overline{\Omega}$ and the sound-soft Dirichlet boundary condition 
$u=0$ on $\Gamma:=\partial\Omega$, with the scattered field $u^s=u-u^i$ outgoing at infinity (i.e., satisfying the Sommerfeld radiation condition - see, e.g., \cite[(2.9)]{ChGrLaSp:11}). We assume that lengths have been non-dimensionalised with respect to a typical length scale of the scatterer (e.g.\ its diameter) so that the wavenumber $k>0$ is non-dimensional. 
Example plots of the solution $u$ for a particular choice of $\Omega$ are shown in Figure~\ref{fig:total1}.

%
%
\subsection{Boundary integral equation formulation}
\label{sec:BIEFormulation}

To solve the boundary value problem described above using a BEM one first writes the solution $u$ using Green's representation theorem (see \cite{Convex} and \cite[(2.107)]{ChGrLaSp:11}) as
\begin{align}
  \label{RepThm}
  u(\bx) = u^i(\bx) - \int_\Gamma \Phi_k(\bx,\by)\pdone{u}{\bn(\by)}(\by)\,\rd s(\by), \qquad \bx\in D,
\end{align}
where $\Phi_k(\bx,\by)=(\ri/4)H^{(1)}_0\left(k\left|\bx-\by\right|\right)$ is the fundamental solution for~(\ref{eqn:HE}), with $H_0^{(1)}(z) = J_0(z) + \ri Y_0(z)$ denoting the Hankel function of the first kind of order zero, and $\pdonetext{u}{\bn}$ is the normal derivative of $u$ on $\Gamma$, with $\bn$ the unit normal directed into $D$. 
As is reviewed in \cite[\S2]{ChGrLaSp:11}, 
from the representation formula \eqref{RepThm} one can derive various boundary integral equations (BIEs) for $\pdonetext{u}{\bn}\in L^2\left(\Gamma\right)$, each taking the form
\begin{align}
  \cA\frac{\partial u}{\partial \bn}&= l, \label{eqn:BIE}
\end{align}
where $l\in L^2\left(\Gamma\right)$ is defined in terms of the incident wave $u^i$, and $\cA:L^2\left(\Gamma\right)\rightarrow L^2\left(\Gamma\right)$ is a bounded linear integral operator. 
The particular choice of operator $\cA$ is irrelevant for the purposes of this paper. However, we note that for the scatterer $\Omega$ under consideration (in fact for all star-like Lipschitz scatterers) it is possible (for details see \cite{SCWGS11}) to choose a formulation in which $\cA$ is \emph{coercive}, satisfying an estimate of the form
\begin{align}
  \big|\left( \cA\psi,\psi\right)_{L^2\left(\Gamma\right)}\big|\geq \gamma \left\|\psi\right\|_{L^2(\Gamma)}^2, \quad \psi\in L^2\left(\Gamma\right),\,k>0,
\label{eqn:Coercivity}
\end{align}
where $\gamma>0$ is a constant independent of $k$, and $\left( \cdot,\cdot\right)_{L^2\left(\Gamma\right)}$ denotes the usual inner product in $L^2(\Gamma)$. By the standard Lax Milgram lemma the coercivity property \rf{eqn:Coercivity}, combined with the boundedness of $\cA$, implies unique solvability of the BIE \rf{eqn:BIE}. Moreover, for any closed subspace $V_N$ of $L^2(\Gamma)$ the Galerkin variational problem,
\begin{align}
\label{eqn:Galerkin}
\textrm{find } v_N\in V_N \textrm{ such that } \left(\cA v_N,w_N\right)_{L^2\left(\Gamma\right)} = \left( l,w_N\right)_{L^2\left(\Gamma\right)}, \textrm{ for all } w_N\in V_N,
\end{align}
also has a unique solution $v_N$ which satisfies the quasi-optimality estimate
\begin{align}
\label{eqn:QuasiOptimality}
\left\|\pdone{u}{\bn}-v_N\right\|_{L^2\left(\Gamma\right)} \leq \frac{\|\cA\|_{L^2\left(\Gamma\right)\to L^2\left(\Gamma\right)} }{\gamma} \inf_{w_N\in V_N} \left\|\pdone{u}{\bn}-w_N\right\|_{L^2\left(\Gamma\right)},
\end{align}
where the infimum on the right-hand side represents the best approximation error for approximating $\pdonetext{u}{\bn}$ by an element of $V_N$. 
 
In the BEM one takes $V_N$ to be a certain finite-dimensional subspace of $L^2(\Gamma)$ (of dimension $N\in\N$, say), which, after choosing a suitable basis for $V_N$, reduces the solution of \rf{eqn:Galerkin} to the inversion of a linear system, in which the matrix entries are integrals involving the basis functions and the operator $\cA$. 
The conventional choice for $V_N$ is a space of piecewise polynomials on $\Gamma$ defined on an appropriate mesh. As explained in \S\ref{sec:Intro}, this choice typically requires $N$ to grow at least in proportion to $k$ as $k\to\infty$ in order to keep the best approximation error fixed. The HNA approach aims to reduce the value of $N$ required when $k$ is large, by using an approximation space $V_N$ consisting of piecewise polynomials multiplied by certain oscillatory functions, which are chosen based on partial knowledge of the high frequency ($k\to\infty$) asymptotic behaviour of $\pdonetext{u}{\bn}$, which we now consider.

\subsection{High frequency solution behaviour}
\label{sec:HFBehaviour}
For the scatterer $\Omega$ of Figure~\ref{BEMmodelproblem}, 
we classify the sides of $\Omega$ into `convex' sides ($\Gamma_{\rm c}$ and $\Gamma_{\rm c}'$) and `nonconvex' sides ($\Gamma_{\rm nc}$ and $\Gamma_{\rm nc}'$), the nature of the high frequency behaviour being different on each type of side.  Denote the lengths of the sides $\Gamma_{\rm c}$, $\Gamma_{\rm c}'$, $\Gamma_{\rm nc}$ and $\Gamma_{\rm nc}'$ by $L_{\rm c}$, $L_{\rm c}'$, $L_{\rm nc}$ and $L_{\rm nc}'$ respectively. On each side let $s$ be the local arc length along the side, measured anti-clockwise around $\Gamma$, as illustrated for the side $\Gamma_{\rm nc}$ in Figure~\ref{BEMmodelproblem}. 

On the convex side $\Gamma_{\rm c}$ (for $\Gamma_{\rm c}'$ things are similar), we have the decomposition (\cite[Theorem 3.2]{NonConvex} and \cite[\S3]{HeLaMe:11})
\begin{align}
  \pdone{u}{\bn}(\bx(s)) = \Psi(\bx(s)) + v^+(s)\re^{\ri ks} +v^-(L_{\rm c}-s) \re^{-\ri ks}, \qquad s\in[0,L_{\rm c}],
\label{Decomp}
\end{align}
where $\Psi$ is the GO approximation (representing the contribution of the incident and specularly reflected waves), and the second and third terms in~\rf{Decomp} represent the combined contribution of all the diffracted waves emanating from the corners $\bS$ and $\bR$ respectively (including the high-order multiply-diffracted waves which have travelled arbitrarily many times around the boundary). 
Explicitly, $\Psi = 2\pdonetext{ u^i}{\bn}$ if $\Gamma_{\rm c}$ is illuminated by the incident wave (i.e.\ if $\bd^i\cdot \bn<0$ on $\Gamma_{\rm c}$) and $\Psi=0$ otherwise. The functions $v^\pm(s)$ are analytic in the right half-plane $\real{s}>0$, and for $k\geq k_0>0$ there exists a constant $C>0$ depending only on $k_0$ and on the shape of $\Omega$ (i.e.\ on its corner angles and not on its size) such that
\begin{align}
  \label{vjpmBounds}
  |v^\pm(s)|\leq \begin{cases} C k^2|ks|^{-\delta^\pm}, & 0<|s|\leq 1/k,\\
  C k^2|ks|^{-1/2}, & |s|> 1/k,
  \end{cases}
  \quad \real{s}>0,
\end{align}
where $\delta^\pm \in (0,1/2)$ depend on the corner angles at $\bS$ and $\bR$ respectively. \footnote{The $k^2$ factor on the right-hand side of \rf{vjpmBounds} can be sharpened to $k^{3/2}\log^{1/2}(2+k)$, and possibly even to $k$ (see \cite[Remark 3.3]{NonConvex}), but we are deliberately keeping things simple here. Similar statements apply to the bounds \rf{tildevEst}, \rf{eqn:BestApprox} and \rf{eqn:BestApprox2} below.} 
The bounds \rf{vjpmBounds} imply that the functions $v^\pm(s)$ are non-oscillatory (i.e.\ slowly varying), since by the Cauchy integral formula one can show that their derivatives grow no faster that the functions themselves with respect to increasing $k$ (cf.\ \cite[Remark 4.2]{ScreenBEM}). They are hence much easier to approximate than $\pdonetext{u}{\bn}$ itself when $k$ is large; this is exploited in the design of the HNA approximation space, which we discuss in the next section.

On the nonconvex side $\Gamma_{\rm nc}$ (for $\Gamma_{\rm nc}'$ things are similar) the asymptotics are more complicated. With regard to the GO components, as well as the incident wave (and its reflection in $\Gamma_{\rm nc}$) we now also expect (depending on the incident angle) a reflected wave 
\begin{align}
u^r(\bx)&=\re^{\ri k \bx\cdot \bd^r},
\qquad \bd^r=(-d^i_1,d^i_2),
\label{eqn:urDef}
\end{align}
generated by the other nonconvex side $\Gamma_{\rm nc}'$ (and also its subsequent reflection in $\Gamma_{\rm nc}$). With regard to diffracted waves, as before we expect waves oscillating like $\re^{\pm \ri ks}$, but we also expect waves oscillating like $\re^{\ri k r(s)}$ (with $r(s)=\sqrt{s^2+L_{\rm nc}'^2}$ as in Figure~\ref{BEMmodelproblem}), corresponding to diffraction from the corner $\bR$. 
Crucially, we now have the possibility that the incident wave may illuminate only part of the side, rather than all or none of the side, as was the case for $\Gamma_{\rm c}$. Precisely, this partial illumination occurs when $\alpha^i<  \alpha<\pi$, where $\alpha^i = \pi-\tan^{-1}(L_{\rm nc}/L_{\rm nc}')$, in which case the shadow boundary $\theta=\alpha+\pi$ intersects $\Gamma_{\rm nc}$ at $\xSB^i = (-s_{\rm SB}^i,-L_{\rm nc}')$, where $s_{\rm SB}^i=L_{\rm nc}'\tan(\pi-\alpha)$. The portion of $\Gamma_{\rm nc}$ to the left of $\xSB^i$ is illuminated, and the portion to the right of $\xSB^i$ is not (cf.\ Figure~\ref{fig:IncShad}). 
Similarly, the reflected wave generated by $\Gamma_{\rm nc}'$ may illuminate only a part of $\Gamma_{\rm nc}$. This kind of partial illumination occurs when $\pi<  \alpha<\alpha^r$, where $\alpha^r=\pi+\tan^{-1}(L_{\rm nc}/L_{\rm nc}')$, in which case the shadow boundary $\theta=3\pi-\alpha$ intersects $\Gamma_{\rm nc}$ at $\xSB^r = (-s_{\rm SB}^r,-L_{\rm nc}')$, where $s_{\rm SB}^r=L_{\rm nc}'\tan(\alpha-\pi)$. The portion of $\Gamma_{\rm nc}$ to the right of $\xSB^r$ is illuminated, and the portion to the left of $\xSB^r$ is not (cf.\ Figure~\ref{fig:RefShad}). 

As can be seen in the plots of Figure~\ref{fig:total1}, across these shadow boundaries the field does not undergo the sharp jump discontinuity predicted by the GO approximation; rather the diffracted field associated with the corner $\bR$ ensures that the transition from `light' to `dark' occurs smoothly but rapidly. As is well known, this transition behaviour is governed by a canonical solution of the Helmholtz equation involving the Fresnel integral. 
Specifically, one can prove \cite[Theorem 3.6, Lemma 3.5]{NonConvex} that on $\Gamma_{\rm nc}$ 
\begin{align}
\label{dudnGamma2Rep}
\pdone{u}{\bn}(\bx(s))&=\Psi(\bx(s)) + v^+(L_{\rm nc}+s)\re^{\ri ks} +v^-(L_{\rm nc}-s) \re^{-\ri ks}+v(s)\re^{\ri kr(s)}, \qquad s\in [0,L_{\rm nc}],
\end{align}
where
\begin{enumerate}[(i)]
\item
$\Psi=2\pdonetext{u^d}{\bn}$ if $\tfrac{\pi}{2}<\alpha<\tfrac{3\pi}{2},$ and $\Psi=0$ otherwise; 
here 
\begin{align}
\label{eqn:udDef}
u^d(r,\theta,\alpha) =E(r,\theta-\alpha) - E(r,\theta+\alpha),
\end{align}
where 
$E(r,\psi) = \re^{-\ri kr\cos{\psi}}\,{\rm Fr}(-\sqrt{2kr}\cos(\psi/2))$, 
and ${\rm Fr}$ is a Fresnel integral, 
\begin{align}
\label{eqn:FresnelDef}
{\rm Fr}(z)  =\frac{1}{2}\erfc(\re^{-\ri \pi/4}z)= \frac{\re^{-\ri\pi/4}}{\sqrt{\pi}} \int_z^\infty \re^{\ri t^2} \,\rd t ,
\end{align}
the integral being understood in the improper sense;
\item the functions $v^\pm(s)$ are analytic in $\real{s}>0$, and for $k\geq k_0>0$ satisfy \rf{vjpmBounds} with $\delta^+=\delta^- \in(0,1/2)$ depending only on the corner angle at $\bP$ and $C>0$ depending only on $k_0$ and the shape of $\Omega$;
\item the function $v(s)$ is analytic in a complex neighbourhood $U_{\rm nc}$ of $[0,L_{\rm nc}]$ (with $U_{\rm nc}$ being independent of $k$, and depending only on the shape of $\Omega$),
with
\begin{align}
\label{tildevEst}
\left|v(s)\right|&\leq Ck^2, \qquad s\in U_{\rm nc},\, k\geq k_0>0,
\end{align}
where $C>0$ depends only on $k_0$ and the shape of $\Omega$.
\end{enumerate}

\begin{rem}
The function $u^d$ is (cf.\ \cite[\S8.2]{BoSeUs:69}) the exact solution to the canonical problem of diffraction of the plane wave $u^i$ by the infinite knife edge $\left\{(r,0)\,:\,r\in[0,\infty)\right\}$ extending the side $\Gamma_{\rm nc}'$ (cf.\ Figure~\ref{fig:KnifeEdge}) on which homogeneous Dirichlet (sound-soft) boundary conditions are imposed. 
The use of this simple canonical solution, instead of the more obvious (yet complicated) choice of the solution to diffraction of $u^i$ by the infinite wedge extending $\Gamma_{\rm nc}'$ and $\Gamma_{\rm c}$ (cf.\ Figure~\ref{fig:KnifeEdge}) is justified because the two solutions have the same shadow boundary behaviour on $\Gamma_{\rm nc}$: the difference between the two solutions is a circular wave whose amplitude varies slowly across the relevant shadow boundaries - see \cite{Oberhettinger} and the discussion in \cite[\S3.2]{NonConvex}. 
\end{rem}

The first term $\Psi$ in \rf{dudnGamma2Rep} represents a modified geometrical optics approximation; depending on the value of $\alpha$, this includes contributions from the incident wave (via $E(r,\theta-\alpha)$ in \rf{eqn:udDef}), and the reflection of the incident wave in $\Gamma_{\rm nc}'$ (via $E(r,\theta+\alpha)$), these waves being switched on/off smoothly across their respective shadow boundaries by the function $E$. (A number of basic properties of $E(r,\psi)$ are collected in \S\ref{sec:EApprox} below.)

As on a convex side, the power of the decomposition \rf{dudnGamma2Rep} is that the amplitudes $v^\pm(s)$ and $v(s)$ of the remaining terms are non-oscillatory. This fact is exploited in the design of the HNA approximation space in \cite{NonConvex}, as will be reviewed shortly. However, the BEM proposed in \cite{NonConvex} requires the canonical solution $u^d$ to be evaluated analytically. The purpose of this paper is to show that such analytical evaluations are not necessary - the shadow boundary behaviour can be treated by mesh refinement. The design of the new HNA approximation space we propose in \S\ref{sec:HNAApproxSpace}, and the associated best approximation error estimates, are based on the following theorem, which follows from 
the results of \S\ref{sec:EApprox}.
\begin{thm}
\label{thm:PsiDecomp}
For $\tfrac{\pi}{2}< \alpha <\tfrac{3\pi}{2}$ we can decompose
\begin{align}
\label{eqn:PsiDecomp}
\Psi(\bx(s)) =\Psi_{\rm GO}(\bx(s)) + V(s)\re^{\ri k r(s)}, \qquad s\in[0,L_{\rm nc}],
\end{align}
and hence, with $\tilde v(s) = v(s) + V(s)$,
\begin{align}
\label{dudnGamma2Rep2}
\pdone{u}{\bn}(\bx(s))&=\Psi_{\rm GO}(\bx(s)) + v^+(L_{\rm nc}+s)\re^{\ri ks} +v^-(L_{\rm nc}-s) \re^{-\ri ks}+\tilde v(s)\re^{\ri kr(s)}, \quad s\in [0,L_{\rm nc}],
\end{align}
where 
$\Psi_{\rm GO}$ is the classical GO approximation,
\begin{align*}
\label{}
\Psi_{\rm GO}(\bx(s)) = H(s-s_{\rm SB}^i)\pdone{u^i}{\bn}(\bx(s)) - H(s_{\rm SB}^r-s)\pdone{u^r}{\bn}(\bx(s)),
\end{align*}
with $H$ denoting the Heaviside function ($H(x) = 0$ for $x<0$, $H(0)=1/2$, and $H(x)=1$, for $x>0$) and $u^r$ defined as in \rf{eqn:urDef}, and 
\begin{align}
\label{eqn:VDef}
V(s) &= -H(s-s_{\rm SB})g^+(s-s_{\rm SB}) + H(s_{\rm SB}-s)g^-(s_{\rm SB}-s) - g^-(s+s_{\rm SB}),
\end{align}
where $s_{\rm SB}=L_{\rm nc}'|\tan(\pi-\alpha)|\geq 0$ and, with the function $g(s;R,\beta)$ defined as in \rf{eqn:gDef}, 
\begin{align*}
\label{}
g^\pm(s) = g(s;L_{\rm nc}'/\cos(\pi-\alpha),\tfrac{\pi}{2}\pm|\pi-\alpha|).
\end{align*}

For any $0<\delta<1$, $g^\pm(t)$ are analytic and bounded in
\begin{align*}
\cS^\delta=\left\lbrace s\in\C:\, \left|\im{s}\right|<(1-\delta)L_{\rm nc}' \textrm{ and } \left|\arg{s}\right|< \arctan\sqrt{(11+5\sqrt{5})/2}\right\rbrace;
\end{align*}
specifically, for any $k_0>0$ there exists $C>0$, depending only on $\delta$ and $k_0L_{\rm nc}'$, such that 
\begin{align}
\label{}
|g^\pm(s))| \leq Ck, \qquad s\in\cS^\delta, \; k\geq k_0>0.
\end{align}
\end{thm}
\begin{proof}
Lemma~\ref{lem:dEdnDecomp} and Theorem~\ref{cor:gReg} together give \rf{eqn:PsiDecomp} with
\begin{align*}
\label{}
V(s) &= -H(s-s_{\rm SB}^i)g(s-s_{\rm SB}^i;R_\alpha,\tfrac{3\pi}{2}-\alpha) 
+ H(s_{\rm SB}^i-s)g(s_{\rm SB}^i-s;R_\alpha,\alpha-\tfrac{\pi}{2})\notag \\
& \qquad - H(s-s_{\rm SB}^r)g(s-s_{\rm SB}^r;R_\alpha,\alpha-\tfrac{\pi}{2}) 
+ H(s_{\rm SB}^r-s)g(s_{\rm SB}^r-s;R_\alpha,\tfrac{3\pi}{2}-\alpha),
\end{align*}
where $R_\alpha=L_{\rm nc}'/\cos{(\pi-\alpha)}$. Then \rf{eqn:VDef} follows from the fact that $s_{\rm SB}^i = -s_{\rm SB}^r$. The decomposition \rf{dudnGamma2Rep2} follows from combining \rf{eqn:PsiDecomp} with \rf{dudnGamma2Rep}.
\end{proof}

\subsection{HNA approximation space} 
\label{sec:HNAApproxSpace}
In this section we outline two different HNA approximation spaces for $\pdonetext{u}{\bn}$ on $\Gamma$, both based on the decompositions stated in the previous section, combined with an $hp$ approximation strategy.  We first review the approach of \cite{NonConvex}, which is based on the decompositions \rf{Decomp} and \rf{dudnGamma2Rep}, with the canonical solution $u^d$ appearing in \rf{dudnGamma2Rep} being evaluated analytically. We then explain our new approach, which uses \rf{dudnGamma2Rep2} in place of \rf{dudnGamma2Rep} (so that no evaluations of $u^d$ are required), the shadow boundary effects being dealt with by mesh refinement.

The graded meshes and the associated spaces of piecewise polynomials we consider below are all constructed from the same basic geometric mesh $\mathcal{G}_n(0,L)$, $L>0$, $n\in\N$, defined in Appendix~\ref{app:meshes}, which is then scaled, reflected and translated as required. 
For simplicity (and for consistency with \cite{NonConvex}) we shall assume the same number of layers $n$ in every geometric mesh in our approximation spaces, and the same maximum polynomial degree $p\in\N_0$ on every mesh element\footnote{For reasons of efficiency and conditioning it is preferable to reduce the maximum polynomial degree near points of mesh refinement, as in, e.g., \cite{ScreenBEM}, but we do not consider this here.}. We also assume throughout - this is our $hp$ assumption - that
\begin{align}
\label{eqn:cDef}
n \geq \max\{1,cp\}, \qquad \textrm{for some fixed constant } c>0.
\end{align}

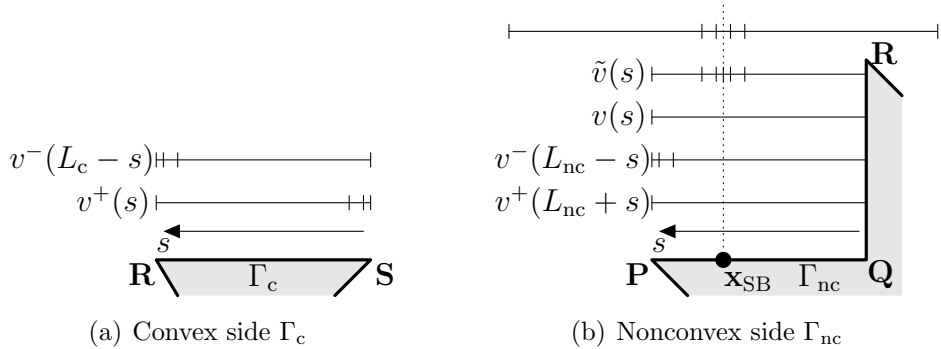
\begin{figure}[t]
  \begin{center}
\subfigure[Convex side $\Gamma_{\rm c}$ \label{fig:MeshConvex}]{
  	\begin{tikzpicture}[line cap=round,line join=round,>=triangle 45,x=1.0cm,y=1.0cm, scale=0.95]
\begin{scope}[xscale=-1,yscale=1,xshift=-3cm,yshift=-0.4cm]
	\fill[line width=0pt,color=cqcqcq,fill=cqcqcq] (0.5,-0.5)--(0,0)--(3,-0) -- (2.7,-0.5)  -- cycle;
	\draw [line width=1.2pt] (0.5,-0.5)--(0,0)--(3,-0) -- (2.7,-0.5);
\end{scope}
	\draw (0,0.4)--(3,0.4);
	\draw (0,0.3)--(0,0.5);
	\draw (2.7,0.3)--(2.7,0.5);
	\draw (2.9,0.3)--(2.9,0.5);
	\draw (3,0.3)--(3,0.5);
	\draw (-0.6,0.4) node {$v^+(s)$};
\begin{scope}[xscale=-1,yscale=1,xshift=-3cm,yshift=0.6cm]
	\draw (0,0.4)--(3,0.4);
	\draw (0,0.3)--(0,0.5);
	\draw (2.7,0.3)--(2.7,0.5);
	\draw (2.9,0.3)--(2.9,0.5);
	\draw (3,0.3)--(3,0.5);
\end{scope}
	\draw (-1.04,1) node {$v^-(L_{\rm c}-s)$};
	\draw [<-] (0.1,0) -- (2.9,0);
	\draw (0.1,-0.2) node {$s$};
	\draw (-0.2,-0.6) node {$\bR$};
	\draw (3.2,-0.6) node {$\bS$};
	\draw (1.5,-0.65) node {$\Gamma_{\rm c}$};
\end{tikzpicture}
}
\hs{5}
\subfigure[Nonconvex side $\Gamma_{\rm nc}$ \label{fig:MeshNonconvex}]{
	\begin{tikzpicture}[line cap=round,line join=round,>=triangle 45,x=1.0cm,y=1.0cm, scale=0.95]
	\begin{scope}[xscale=1,yscale=1,yshift=-0.4cm]
		\fill[line width=0pt,color=cqcqcq,fill=cqcqcq]  (5.5,2.3)--(5,2.8)--  (5,0)-- (2,0)-- (2.5,-0.5) --(5.5,-0.5)  -- cycle;
	\draw [line width=1.2pt] (5.5,2.3)--(5,2.8)--  (5,0)-- (2,0)-- (2.5,-0.5);
	\def\sSB{3}
	\filldraw (\sSB,0) circle (3pt);
	\draw (\sSB+0.35,-0.3) node {$\xSB$};
	\draw[dotted] (\sSB,0)--(\sSB,3.6);
	\end{scope}
	\draw [<-] (2.1,0) -- (4.9,0);
	\draw (2.1,-0.2) node {$s$};
\begin{scope}[xscale=-1,yscale=1,xshift=-5cm,yshift=0.6cm]
	\draw (0,0.4)--(3,0.4);
	\draw (0,0.3)--(0,0.5);
	\draw (2.7,0.3)--(2.7,0.5);
	\draw (2.9,0.3)--(2.9,0.5);
	\draw (3,0.3)--(3,0.5);
\end{scope}
	\draw (0.9,0.4) node {$v^+(L_{\rm nc}+s)$};
\begin{scope}[xscale=-1,yscale=1,xshift=-5cm,yshift=0cm]
	\draw (0,0.4)--(3,0.4);
	\draw (0,0.3)--(0,0.5);
	\draw (3,0.3)--(3,0.5);
\end{scope}
	\draw (0.9,1) node {$v^-(L_{\rm nc}-s)$};
\begin{scope}[xscale=-1,yscale=1,xshift=-5cm,yshift=1.2cm]
	\draw (0,0.4)--(3,0.4);
	\draw (0,0.3)--(0,0.5);
	\draw (3,0.3)--(3,0.5);
\end{scope}
\begin{scope}[xshift=2cm,yshift=1.8cm]
	\draw (0,0.4)--(3,0.4);
	\draw (0,0.3)--(0,0.5);
	\draw (1,0.3)--(1,0.5);
	\draw (1.1,0.3)--(1.1,0.5);
	\draw (1.3,0.3)--(1.3,0.5);
	\draw (0.9,0.3)--(0.9,0.5);
	\draw (0.7,0.3)--(0.7,0.5);
\end{scope}
\begin{scope}[xshift=2cm,yshift=2.4cm]
	\draw (-2,0.4)--(4,0.4);
	\draw (-2,0.3)--(-2,0.5);
	\draw (4,0.3)--(4,0.5);
	\draw (1.1,0.3)--(1.1,0.5);
	\draw (1.3,0.3)--(1.3,0.5);
	\draw (0.9,0.3)--(0.9,0.5);
	\draw (0.7,0.3)--(0.7,0.5);
	\draw (-0.2,-3) node {$\bP$};
	\draw (3.2,-3) node {$\bQ$};
	\draw (3.25,0.1) node {$\bR$};
	\draw (2.35,-3.05) node {$\Gamma_{\rm nc}$};
\end{scope}
	\draw (1.54,1.6) node {$v(s)$};
		\draw (1.54,2.2) node {$\tilde{v}(s)$};
\end{tikzpicture}
}
\end{center}
  \caption{Illustration of the overlapping meshes. For the approximation on $\Gamma_{\rm nc}$, in \cite{NonConvex} the decomposition \rf{dudnGamma2Rep} is used in conjunction with the lower three meshes in (b), the shadow boundary behaviour being computed analytically through $u^d$. In this paper we use \rf{dudnGamma2Rep2} instead of \rf{dudnGamma2Rep}, replacing the mesh for $v(s)$ (comprising a single element on the whole side) by the mesh for $\tilde v(s)$, which is graded towards the point $\xSB$ where the shadow boundary intersects $\Gamma_{\rm nc}$.
\label{fig:meshes}
}
\end{figure}

The approximation space $V_N\subset L^2(\Gamma)$ in \cite{NonConvex} is defined as follows. On the convex side $\Gamma_{\rm c}$ ($\Gamma_{\rm c}'$ is similar) it uses \rf{Decomp}, with $\Psi$ evaluated analytically and $v^+(s)$ and $v^-(L_{\rm c}-s)$ approximated by piecewise polynomials of maximum polynomial degree $p\in \N_0$ on overlapping geometrically graded meshes refined towards the corner singularities at $\bS$ ($s=0$) and $\bR$ ($s=L_{\rm c}$) respectively (see Figure~\ref{fig:MeshConvex}). 
On $\Gamma_{\rm nc}$ ($\Gamma_{\rm nc}'$ is similar) it uses \rf{dudnGamma2Rep} with $\Psi$ evaluated analytically (e.g.\ using the algorithm of \cite{AlChLP:13}) and $v^+(L_{\rm nc} +s)$, $v^-(L_{\rm nc}-s)$ and $v(s)$ approximated by piecewise polynomials of degree $\leq p$ on overlapping meshes. Specifically, $v^-(L_{\rm nc}-s)$ requires a mesh graded towards the singularity at $\bP$ ($s=L_{\rm nc}$), but $v^+(L_{\rm nc}+s)$ and $v(s)$ are regular enough\footnote{The lack of a singularity at $\bQ$ ($s=0$) is due to the fact that the exterior angle at $\bQ$ equals $\pi$ divided by an integer; polygons with more general `nonconvex' angles would require mesh refinement towards $\bQ$ too - see the discussion in \cite[\S8]{NonConvex}).} to each require only a single polynomial of degree $\leq p$ supported on the whole side (see Figure~\ref{fig:MeshNonconvex}). 
In \cite[Theorem 5.6]{NonConvex} the following best approximation error estimate is proved for 
$V_N$.
\begin{align}
\label{eqn:BestApprox}
\inf_{w_N\in V_N} \left\|\pdone{u}{\bn}-w_N\right\|_{L^2\left(\Gamma\right)} \leq Ck^2 \re^{-p\tau}, \qquad k\geq k_0>0,
\end{align}
where $C>0$ depends only on $k_0$, $\sigma$ (the grading parameter in the definition of $\mathcal{G}_n(0,L)$) and the shape of $\Omega$, and $\tau>0$ depends only on $c$ (the constant in \rf{eqn:cDef}), $\sigma$ and the shape of $\Omega$. 
The best approximation error therefore decays exponentially as the maximum polynomial degree $p$ increases, as one would expect from an $hp$ method.  
Moreover, since the total number of degrees of freedom $N$ is approximately proportional to $p^2$ (precisely, $N = (6n+4)(p+1)$ for the particular $\Omega$ under consideration), \rf{eqn:BestApprox} implies that increasing $N$ in proportion to $\log^2{k}$ as $k\to\infty$ is sufficient to ensure that the best approximation error stays bounded (cf.~the discussion in \cite[Remark 6.5]{HeLaMe:11}). This represents a significant saving over conventional BEMs in which the full solution $\pdonetext{u}{\bn}$ is approximated using piecewise polynomials, which generally require $N$ to grow at least linearly with respect to increasing $k$.



We now show that similar performance can be achieved without the need to evaluate the canonical solution $u^d$ in $\Psi$ analytically. 
The modified approximation space $\widetilde V_N\subset L^2(\Gamma)$ we propose is defined as follows. On the convex sides $\Gamma_{\rm c}$ and $\Gamma_{\rm c}'$ we use the same approximation based on \rf{Decomp} as in $V_N$. On the nonconvex side $\Gamma_{\rm nc}$ ($\Gamma_{\rm nc}'$ is similar) we use the decomposition \rf{dudnGamma2Rep2} instead of \rf{dudnGamma2Rep}. The first term $\Psi_{\rm GO}$ in \rf{dudnGamma2Rep2}, representing the classical GO approximation, with the incident and reflected waves $u^i$ and $u^r$ cut off sharply across their respective shadow boundaries, is evaluated analytically; this requires only the evaluation of plane waves and no other special functions. The second and third terms $v^+(L_{\rm nc} +s)$ and $v^-(L_{\rm nc}-s)$ are approximated precisely as in $V_N$. The final term $\tilde v(s)$ is approximated on a mesh $\mathcal{M}_{\rm nc}$ graded towards the shadow boundary.\footnote{We remark that the functions $g^\pm(s)$ are not singular at $s=0$, but they vary rapidly near $s=0$, as is reflected by the fact that they are not bounded with a $k$-independent bound on any $k$-independent neighbourhood of $s=0$ (see Remark~\ref{OriginRem} below). We approximate $\tilde v(s)$ on a graded mesh in order to capture this rapid variation.}  Specifically, we propose the choice
\begin{align*}
\label{}
\mathcal{M}_{\rm nc} = \big\lbrace\big\lbrace\{s_{\rm SB} + \mathcal{G}_n(0,L_{\rm nc})\}\cup \{s_{\rm SB} - \mathcal{G}_n(0,L_{\rm nc})\}\big\rbrace \cap (0,L_{\rm nc})\big\rbrace \cup \{0,L_{\rm nc} \}.
\end{align*}

Let $\mathcal{P}^p_{{\rm nc}}$ denote the space of piecewise polynomials of degree $\leq p$ on the mesh $\mathcal{M}_{\rm nc}$. 
Standard $hp$ approximation arguments (similar to those in, e.g., \cite{HeLaMe:11,ConvexPreprint,NonConvex}) based on Lemma~\ref{EllipseLem} and the regularity results in Theorem~\ref{thm:PsiDecomp}, show that 
\begin{align}
\label{eqn:BestApproxtildev}
\inf_{q\in \mathcal{P}^p_{\rm nc}} \left\|V - q\right\|_{L^2\left(\Gamma_{\rm nc}\right)} \leq Ck \re^{-p\tau}, \qquad k\geq k_0>0,
\end{align}
where $C>0$ depends only on $k_0$, $\sigma$ and the shape of $\Omega$, and $\tau>0$ depends only on $c$, $\sigma$ and the shape of $\Omega$. Numerical experiments validating \rf{eqn:BestApproxtildev} are described in \S\ref{sec:NumResults}.

Combining \rf{eqn:BestApproxtildev} (and an analogous result for $\Gamma_{nc}'$) with \rf{eqn:BestApprox} gives the following theorem, which holds not just for the particular polygon $\Omega$ shown in Figure~\ref{fig:Geometry} but also, after appropriate obvious modifications to $\widetilde V_N$, for all of the polygons in the class defined in \cite[Definition 3.1]{NonConvex}. 
\begin{thm}
\label{thm:BestApprox}
The best approximation error for approximating $\pdonetext{u}{\bn}$ in $\widetilde V_N$ satisfies
\begin{align}
\label{eqn:BestApprox2}
\inf_{w_N\in \widetilde V_N} \left\|\pdone{u}{\bn}-w_N\right\|_{L^2\left(\Gamma\right)} \leq Ck^2 \re^{-p\tau}, \qquad k\geq k_0>0,
\end{align}
where $C>0$ depends only on $k_0$, $\sigma$ and the shape of $\Omega$, and $\tau>0$ depends only on $c$, $\sigma$, and the shape of $\Omega$. 
\end{thm}
We end this section by remarking that, by the quasi-optimality result \rf{eqn:QuasiOptimality}, the error bound \rf{eqn:BestApprox2} in Theorem~\ref{thm:BestApprox} implies similar exponential convergence results for the Galerkin approximation to $\pdonetext{u}{\bn}$, the resulting approximation to the solution $u$ in the domain (as computed using the representation formula \rf{RepThm}), and also the far field pattern - for details see \cite[\S6]{NonConvex}.

\section{Approximation properties of $E(r,\psi)$}
\label{sec:EApprox}

In this section we study certain approximation properties of the function 
\begin{align*}
E(r,\psi) = \re^{-\ri kr\cos{\psi}}\,{\rm Fr}(-\sqrt{2kr}\cos\left(\psi/2\right)),
\end{align*}
where $(r,\psi)$ are polar coordinates and ${\rm Fr}$ is the Fresnel integral defined in \rf{eqn:FresnelDef}, which are needed to prove the results in \S\ref{sec:polygons}.

The function $E(r,\psi)$ is $4\pi$-periodic in $\psi$, and satisfies the relations
\begin{alignat}{2}
\label{eqn:ESymmetry1}
E(r,-\psi) &= E(r,\psi), &\qquad \psi\in\R,\\
\label{eqn:ESymmetry2}
E(r,\psi+2\pi) &= \re^{-\ri k r\cos{\psi}} - E(r,\psi), & \qquad \psi\in\R,
\end{alignat}
the latter following from the well-known symmetry of the Fresnel integral,
\begin{align}
\label{eqn:FrSymmetry}
\Fr(-z) = 1-\Fr(z), \qquad z\in\C.
\end{align}
Physically, $E$ represents the exact solution to the diffraction of an incident plane wave by an infinite knife edge aligned with the line $\psi=0$ on which homogeneous Neumann (sound-hard) boundary conditions are imposed, with the incident wave propagating along the top side $\psi=0^+$ of the knife edge and parallel to it (see Figure~\ref{KnifeEdge}). This problem is in a sense the `simplest' canonical edge diffraction problem, because there is no specularly reflected wave. 
There is a single shadow boundary at $\psi=\pi$, the incident wave being present in $\psi<\pi$ and not in $\psi>\pi$. 
Indeed, $E$ satisfies the GTD-type approximation (cf.\ e.g.\ \cite[\S7]{DLMF})
\begin{align}
\label{eqn:EAsympt}
E(r,\psi) \sim
\begin{cases}
\re^{-\ri k r\cos\psi} + d(\psi)\frac{\re^{\ri k r}}{\sqrt{kr}}\left(1+\ord{\frac{1}{kr}}\right), & \psi \in [(4n+1)\pi+\delta,(4n+3)\pi-\delta],\\
d(\psi)\frac{\re^{\ri k r}}{\sqrt{kr}}\left(1+\ord{\frac{1}{kr}}\right), & \psi \in [(4n-1)\pi+\delta,(4n+1)\pi-\delta],
\end{cases}
\end{align}
as $kr\to \infty$, where $d(\psi) = -\re^{\ri \pi/4}/(2 \sqrt{2\pi}\cos{(\psi/2)})$, $n\in\Z$, $0<\delta<\pi$ is arbitrary and the approximations hold uniformly in $\psi$ in the stated intervals. The term $\re^{-\ri k r\cos\psi}$ represents a plane wave propagating from the direction $\psi=0$, and $d(\psi)\re^{\ri k r}/\sqrt{kr}$ represents a circular wave emanating from $r=0$ with directionality $d(\psi)$ (a diffraction coefficient). 
The approximation \rf{eqn:EAsympt} is invalid near the critical angles $\psi=(2n+1)\pi$ for $n\in\Z$, with the function $d(\psi)$ blowing up at these values of $\psi$. 

\begin{figure}[t]
\begin{center}
\begin{tikzpicture}[line cap=round,line join=round,>=triangle 45,x=1.0cm,y=1.0cm, scale=2.5]
\draw [ultra thick] (0,0) -- (2,0);
\draw [->] (0,0) -- (1,0);
\draw (1,-0.1) node {$y_1$};
\draw [->] (0,0) -- (0,1);
\draw (0.12,0.9) node {$y_2$};
\draw [dashed] (-2,0) -- (0,0);
\draw [<->] (-1.5,-0.07) -- (0,-0.07);
\draw (-0.74,-0.2) node {$R$};
\filldraw (-1.5,0) circle (1pt);
\draw (-1.32,-0.3) -- ++ (-0.8,0.8*1.73);
\draw (-2.2,1.1) node {$\mathcal{L}$};
\draw [<->] (-1.5-0.04*1.73,-0.04) -- ++ (-0.4*1,0.4*1.73);
\draw (-1.85,0.3) node {$s$};
\filldraw (-1.5-0.4,0+0.4*1.73) circle (1pt);
\draw (-1.85,0.8) node {$\by$};
\draw [thick,->] (-1.5-0.5,0+0.5*1.73) -- ++ (0.2*1.73,0.2) ;	
\draw (-1.8,1.1) node {$\bn$};
\draw [<->] (-0.05,0.05) -- (-1.5-0.35,0+0.4*1.73);
\draw (-0.95,0.5) node {$r$};
\draw [->] (0.4,0) arc (0:155:0.4);
\draw (0.35,0.4) node {$\psi$};
\draw (-1.3,0) arc (0:120:0.2);
\draw (-1.22,0.18) node {$\beta$};
\draw (1.5,0) -- (1.5,0.8);
\draw (1.4,0) -- (1.4,0.8);
\draw (1.6,0) -- (1.6,0.8);
\draw [<-] (1.2,0.4) -- (1.7,0.4);
\draw (1.84,0.55) node {$\re^{-\ri ky_1}$};
\end{tikzpicture} 
\caption{Geometry of diffraction by a knife edge (thick line).}
\label{KnifeEdge}
\end{center}
\end{figure}
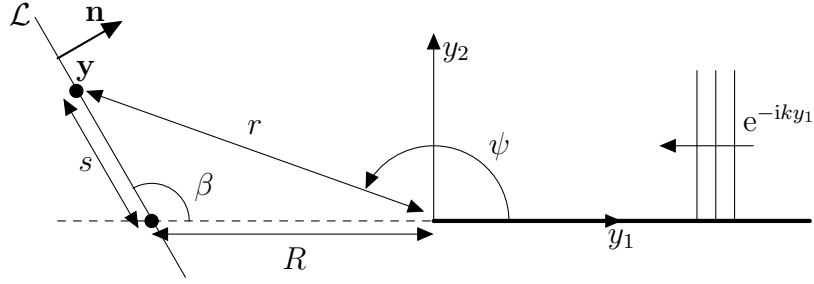 

We shall investigate the approximation properties of the function $E$ (and $\pdonetext{E}{\bn}$) on the line $\mathcal{L}$ defined in the Cartesian coordinates of Figure~\ref{KnifeEdge} by $\by=(y_1,y_2) = (-R,0)+s(\cos{\beta},\sin{\beta})$, where $R>0$, $0<\beta<\pi$ and $s\in\R$. In the context of the application considered in \S\ref{sec:polygons}, namely scattering by the polygon $\Omega$ of Figure~\ref{fig:Geometry}, a portion of the line $\mathcal{L}$ is to be identified with the nonconvex side $\Gamma_{\rm nc}$, after a suitable coordinate rotation. Using \rf{eqn:ESymmetry1}-\rf{eqn:ESymmetry2} it is easy to verify that when the approximation \rf{eqn:EAsympt} is applied in \rf{eqn:udDef}, the non-uniformity in \rf{eqn:EAsympt} manifests itself exactly at the incident and reflected shadow boundaries $\theta=\alpha$ and $\theta=3\pi-\alpha$ respectively.

The following lemma, which follows trivially from \rf{eqn:ESymmetry1}-\rf{eqn:ESymmetry2}, 
shows that $E$ can be written as a sum of the classical GO approximation (comprising the incident wave, cut off sharply across the shadow boundary), and a remainder term proportional to $\re^{\ri kr}$.
\begin{lem}
\label{lem:EDecomp}
The function $E$ can be decomposed as
\begin{align}
\label{eqn:EDecomp}
E(r,\psi) = E_{\rm GO}(r,\psi)
-\sign{(\pi-\psi)}F(\mu(r,\psi)) \re^{\ri kr},
\qquad r>0,\;\psi\in(0,2\pi),
\end{align}
where
\begin{align}
\label{}
E_{\rm GO}(r,\psi)=H(\pi-\psi)\re^{-\ri kr\cos\psi} = H(\pi-\psi)\re^{-\ri ky_1},
\end{align}
\begin{align}
\label{}
\mu(r,\psi) =\sqrt{2kr}\cos{\frac{|\psi|}{2}}
= \sqrt{kr(1+\cos{\psi})} = \sqrt{k(y_1+r)} = \sqrt{\frac{ky_2^2}{r-y_1}}>0,
\end{align}
for $\psi\in(0,\pi)\cup (\pi,2\pi)$, 
with $\mu=0$ for $\psi=\pi$, and
\begin{align}
\label{eqn:gDefn}
F(z)=\re^{-\ri z^2}{\rm Fr}(z) = \frac{1}{2}w(\re^{\ri\pi/4}z),
\end{align}
where $w(z)$ is the (scaled) complementary error function (see e.g. \cite{DLMF}, equation (7.2.3)).

\end{lem}
The function $F(z)$ is entire, and by \rf{eqn:FrSymmetry} satisfies the symmetry relation
\begin{align}
\label{eqn:gSymmetry}
F(-z)=e^{-\ri z^2}-F(z), \qquad z\in\C.
\end{align}
Also, $F(z)$ satisfies the differential equation (cf.\ \cite[(7.10.2)]{DLMF})
\begin{align*}
\label{}
F'(z)=\frac{\re^{\ri 3\pi/4}}{\sqrt{\pi}} -2\ri z F(z), \qquad z\in\C,
\end{align*}
using which one can derive the following decomposition for $\pdonetext{E}{\bn}$. 
\begin{lem}
\label{lem:dEdnDecomp}
On the line $\mathcal{L}$ the normal derivative $\pdonetext{E}{\bn}$ can be decomposed as
\begin{align}
\label{eqn:dEdnDecomp}
\pdone{E}{\bn}(r,\psi) = \pdone{E_{\rm GO}}{\bn}(r,\psi)
-\sign{(\pi-\psi)}G(r,\psi) \re^{\ri kr},
\qquad r>0,\;\psi\in(0,2\pi),
\end{align}
where 
\begin{align}
\label{eqn:vDef}
G(r,\psi) = \frac{\re^{\ri 3\pi/4}}{\sqrt{\pi}}\pdone{\mu}{\bn}(r,\psi) -\ri k \sin\beta\, F(\mu(r,\psi)).
\end{align}
\end{lem}

Our goal is to investigate the approximation properties of 
\begin{align}
\label{eqn:gDef}
g(s;R,\beta):=G(r(s),\psi(s))
\end{align}
as a function of $s\in\R$, for fixed $R>0$ and $0<\beta<\pi$. By symmetry we can without loss of generality restrict attention to $s>0$, since $g(-s;R,\beta)=g(s;R,\pi-\beta)$. Furthermore, Lemma~\ref{lem:dEdnDecomp} reduces the study of $g(s;R,\beta)$ to the study of the functions 
\begin{align}
\label{eqn:fDef}
f(s;R,\beta)&:=F(\mu(s)),
\\
\label{eqn:hDef}
h(s;R,\beta)&:=\pdone{\mu}{\bn}(r(s),\psi(s)).
\end{align}
We begin with the former. We will show that $f(s;R,\beta)$ is analytic and bounded in a $k$-independent complex neighbourhood of the positive real $s$-axis. We start by reviewing some elementary properties of the function $F$.
\begin{lem}
\label{gLem}
There exists a constant $0<C<2$ such that 
\begin{align}
\label{eqn:gBound}
|F(z)|\leq C, \qquad \arg{z}\in[-\tfrac{\pi}{2},\pi].
\end{align}
In the sector $\arg{z}\in(-\pi,\tfrac{\pi}{2})$, $F(z)$ grows exponentially fast as $|z|\to\infty$, with
\begin{align}
\label{eqn:ExpGrowth}
e^{|z|^2\sin{(2\arg{z})}}-\frac{1}{2}\leq |F(z)|\leq e^{|z|^2\sin{(2\arg{z})}}+\frac{1}{2}, \qquad \arg{z}\in(-\pi,-\tfrac{\pi}{2}).
\end{align}
\end{lem}
\begin{proof}
The integral representation \cite[Equation (7.7.2)]{DLMF}
\begin{align}
\label{}
w(z) = \frac{2z}{\pi \ri}\int_0^\infty \frac{\re^{-t^2} \, \rd t}{t^2-z^2}, \qquad \im{z}>0,
\end{align}
implies (by appropriate contour deformations, changes of variable and analytic continuation arguments) the following integral representations for $F$:
\begin{align}
\label{eqn:gIntegralReps}
F(z) = 
\begin{dcases}
\frac{\re^{-\ri \pi/4}}{\pi}\int_0^\infty \frac{\re^{-z^2t^2} \, \rd t}{t^2-\ri}
, & -\tfrac{\pi}{4} \leq \arg{z}\leq\tfrac{\pi}{4},\\
\frac{1}{\pi}\int_0^\infty \frac{\re^{\ri z^2t^2} \, \rd t}{t^2+1}
, & 0 \leq \arg{z}\leq\tfrac{\pi}{2},\\
\frac{\re^{\ri \pi/4}}{\pi}\int_0^\infty \frac{\re^{z^2t^2} \, \rd t}{t^2+\ri}
, & \tfrac{\pi}{4} \leq \arg{z}\leq \tfrac{\pi}{4},
\end{dcases}
\end{align}
from which it follows that
\begin{align}
\label{eqn:gBoundWorking}
|F(z)| \leq 
\begin{dcases}
\frac{1}{\pi}\int_0^\infty \frac{\rd t}{t^2+1} = \frac{1}{2}, & 0 \leq \arg{z}\leq\tfrac{\pi}{2},\\
\frac{1}{\pi}\int_0^\infty \frac{\rd t}{\sqrt{t^4+1}} = \frac{4}{\pi^{3/2}}( \Gamma(5/4))^2,
& -\tfrac{\pi}{4} \leq \arg{z}\leq\tfrac{\pi}{4} \textrm{ or } \tfrac{\pi}{4} \leq \arg{z}\leq \tfrac{3\pi}{4},
\end{dcases}
\end{align}
where $\Gamma(z)$ is the usual Gamma function. Using the symmetry relation \rf{eqn:gSymmetry}, the bound \rf{eqn:gBound} then follows with $C = 1+ (4/\pi^{3/2})( \Gamma(5/4))^2 \approx 1.59$. (This bound is not sharp - numerical evaluations suggest that the optimal constant satisfies $C\approx 1.17$.) The claimed exponential growth in \rf{eqn:ExpGrowth} follows from \rf{eqn:gSymmetry} and the first estimate in \rf{eqn:gBoundWorking}.
\end{proof}

\begin{rem}
A more careful analysis of the integrals in \rf{eqn:gIntegralReps} reveals that $F(z)$ decays like $\ord{1/|z|}$ as $z\to\infty$ in the sector $\arg{z}\in(-\tfrac{\pi}{2},\pi)$, the decay being nonuniform as $\arg{z}$ approaches $-\tfrac{\pi}{2}$ or $\pi$, along which directions $F(z)$ is oscillatory, being asymptotic to $\re^{-\ri z^2}$ as $|z|\to\infty$. However, the simple bound \rf{eqn:gBound} will be sufficient for our purposes in this paper.
\end{rem}


Now, for $s>0$ we have
\begin{align*}
\mu(s) &= \sqrt{k\left(-R +s\cos{\beta} + r(s)\right)},
\end{align*}
where
\begin{align}
\label{innerarg}
r(s)=\sqrt{(-R+s\cos\beta)^2+(s\sin\beta)^2}=\sqrt{R^2+s^2-2sR\cos{\beta}}. 
\end{align}
Equivalently,
\begin{align}
\label{EtasDef2}
\mu(s) =  \frac{\sqrt{k}s\sin{\beta}}{\sqrt{R -s\cos{\beta} + r(s)}},
\end{align}
and we adopt \rf{EtasDef2} as the formula for the analytic continuation of $\mu(s)$ from $s>0$ into the complex $s$-plane, with the complex square roots in \rf{EtasDef2} and \rf{innerarg} taking the principal value. 
The function $r(s)$ has branch points 
at $s=R\re^{\pm \ri \beta}$, with associated branch cuts running from $R\re^{\pm i \beta}$ to $R\cos{\beta}\pm \ri\infty$ respectively (see Figure~\ref{S0}). We denote the resulting cut $s$-plane by
\[\cC_{R,\beta}=\C\setminus\big\lbrace R\cos\beta + \ri \{[R\sin\beta,\infty) \cup (-\infty,-R\sin\beta]\}\big\rbrace.\]
The outer square root in the denominator of \rf{EtasDef2} does not introduce any further branch points; 
in fact, one can show that ${\rm Re}[R -s\cos{\beta} + r(s)]>0$ for all $s\in \cC_{R,\beta}$.
Hence $\mu(s)$ defined by \rf{EtasDef2} is analytic in $\cC_{R,\beta}$. 
The chain rule, combined with Lemma~\ref{gLem} and some tedious but elementary calculations, then implies the following lemma.

\begin{lem}
\label{hlem}
The function $f(s;R,\beta)$ defined by \rf{eqn:fDef}, with $\mu(s)$ defined by \rf{EtasDef2} and $r(s)$ by \rf{innerarg} (with square roots taking principle values), is analytic in the cut plane $\cC_{R,\beta}$. Furthermore, 
\begin{align}
\label{eqn:fBound}
|f(s;R,\beta)|\leq C, \qquad s\in\cR_{R,\beta}\subset \cC_{R,\beta} ,
\end{align}
where $0<C<2$ is the constant from \rf{eqn:gBound} and (see Figure~\ref{S0} for an illustration)
\begin{align}
\label{eqn:RRbetaDef}
\cR_{R,\beta} =
\begin{cases}
\cC_{R,\beta}\cap\left( \mathcal{H}\cup \mathcal{H}_{R,\beta}\cup \mathcal{E}_{R,\beta} \right), & 0<\beta<\tfrac{\pi}{2},\\
\cC_{R,\beta}\cap\left( \mathcal{H}\cup \mathcal{H}_{R,\beta}\right), & \beta=\tfrac{\pi}{2},\\
\cC_{R,\beta}\cap\left( \mathcal{H}\cup \left(\mathcal{H}_{R,\beta}\setminus \mathcal{E}_{R,\beta}\right) \right), & \tfrac{\pi}{2}<\beta<\pi,
\end{cases}
\end{align}
where
\begin{align*}
\label{}
\mathcal{H}= \left\{ s:\im{s}>0\right\},
\qquad
\mathcal{H}_{R,\beta} = \left\{ s:\real{s}>R\cos{\beta}\right\},
\end{align*}
and $\mathcal{E}_{R,\beta}$ is the interior of the ellipse
\begin{align}
\label{EllipseEqn}
\im{s}^2+\frac{\left(\real{s}-R\cos{\beta}\right)^2}{\cos^2{\beta}}=R^2,
\end{align}
which has its foci at $R\re^{\pm \ri \beta}$, semi-minor axis equal to $R\cos{\beta}$, and semi-major axis equal to $R$.
(For $\beta=\tfrac{\pi}{2}$, the ellipse degenerates to the line segment $[-\ri R,\ri R]$ and $\mathcal{E}_{R,\beta}$ is empty.) 

In particular, for $\tfrac{\pi}{2}\leq \beta < \pi$ the region $\cR_{R,\beta}$ contains the right half-plane $H_{\pi/2}=\{s:\real{s}>0\}$. For $0<\beta < \tfrac{\pi}{2}$ the region $\cR_{R,\beta}$ contains the set
\begin{align}
\label{eqn:SemiArrow}
\cS_{R,\beta}= \left\lbrace s\in\C:\,\left|\im{s}\right|<R\sin\beta \textrm{ and } \left|\arg{s}\right|< \arctan\sqrt{(11+5\sqrt{5})/2}\right\rbrace.
\end{align}
\end{lem}

\begin{proof}
Elementary calculations show that $\arg{\mu(s)}\in[-\tfrac{\pi}{2},\pi]$ inside the region $\cR_{R,\beta}\subset \cC_{R,\beta}$ defined in \rf{eqn:RRbetaDef}. The bound \rf{eqn:fBound} then follows from the chain rule and Lemma~\ref{gLem}. For \rf{eqn:SemiArrow}, one observes that for $0<\beta < \tfrac{\pi}{2}$ the boundary of the ellipse $\mathcal{E}_{R,\beta}$ intersects the lines $\im{s}=\pm R\sin\beta$ at the points $s=R\left(\cos\beta(1-\cos\beta) \pm \ri \sin\beta\right)$, and that the ratio $\sin\beta/(\cos\beta(1-\cos\beta))$ is minimised over $0<\beta < \tfrac{\pi}{2}$ at $\beta=2\arctan\sqrt{\sqrt{5}-2}$, with value $\sqrt{(11+5\sqrt{5})/2}$.
\end{proof}

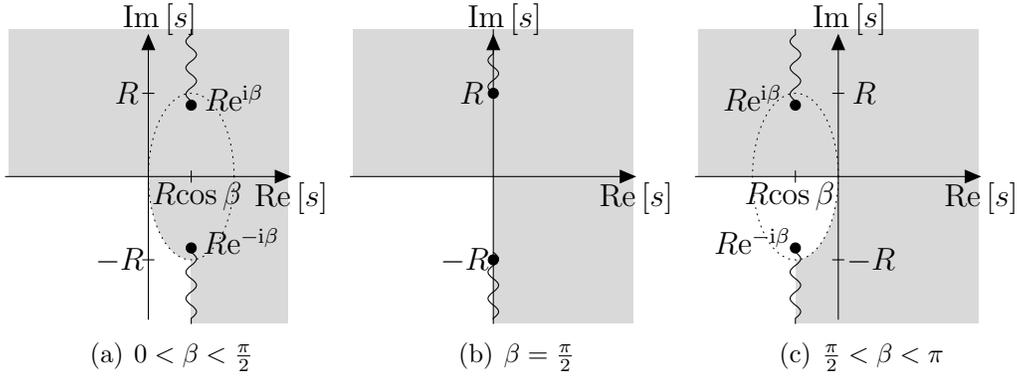
\begin{figure}[t]
\begin{center}
\subfigure[$0<\beta<\tfrac{\pi}{2}$\label{S0a}]{
\begin{tikzpicture}[line cap=round,line join=round,>=triangle 45,x=1.0cm,y=1.0cm, scale=1.9]
\fill[gray!30!white] (0.3,-1.03) -- (0.3,1.03) -- (0.98,1.03) -- (0.98,-1.03);
\fill[gray!30!white] (-0.98,0) -- (-0.98,1.03) -- (0.98,1.03) -- (0.98,0);
\fill[gray!30!white] (0.3,0) ellipse (0.3 and 0.583);
\draw[dotted] (0.3,0) ellipse (0.3 and 0.583);
\draw [->] (-1,0) -- (1,0);
\draw (1,-0.15) node {$\real{s}$};
\draw [->] (0,-1) -- (0,1);
\draw (0.08,1.1) node {$\im{s}$};
\filldraw (0.3,0.5) circle (1pt);
\draw (0.6,0.55) node {$R\re^{\ri\beta}$};
\draw[decorate,decoration={coil,aspect=0,segment length=10pt,amplitude=2pt}] (0.3,0.53) -- (0.3,1.03);
\filldraw (0.3,-0.5) circle (1pt);
\draw (0.65,-0.45) node {$R\re^{-\ri\beta}$};
\draw[decorate,decoration={coil,aspect=0,segment length=10pt,amplitude=2pt}] (0.3,-0.53) -- (0.3,-1.03);
\draw (0.3,-0.04) -- (0.3,0.04) ;
\draw (0.34,-0.14) node {$R\!\cos{\beta}$};
\draw (-0.04,0.583) -- (0.04,0.583) ;
\draw (-0.15,0.583) node {$R$};
\draw (-0.04,-0.583) -- (0.04,-0.583) ;
\draw (-0.2,-0.583) node {$-R$};
\end{tikzpicture} }
\hs{-4}
\subfigure[$\beta=\tfrac{\pi}{2}$\label{S0b}]{
\begin{tikzpicture}[line cap=round,line join=round,>=triangle 45,x=1.0cm,y=1.0cm, scale=1.9]
\fill[gray!30!white] (0,-1.03) -- (0,1.03) -- (0.98,1.03) -- (0.98,-1.03);
\fill[gray!30!white] (-0.98,0) -- (-0.98,1.03) -- (0.98,1.03) -- (0.98,0);
\draw [->] (-1,0) -- (1,0);
\draw (1,-0.15) node {$\real{s}$};
\draw [->] (0,-1) -- (0,1);
\draw (0.08,1.1) node {$\im{s}$};
\filldraw (0,0.583) circle (1pt);
\draw[decorate,decoration={coil,aspect=0,segment length=10pt,amplitude=2pt}] (0,0.583) -- (0,1.03);
\filldraw (0,-0.583) circle (1pt);
\draw[decorate,decoration={coil,aspect=0,segment length=10pt,amplitude=2pt}] (0,-0.53) -- (0,-1.03);
\draw (-0.15,0.583) node {$R$};
\draw (-0.2,-0.583) node {$-R$};
\end{tikzpicture} }
\hs{-4}
\subfigure[$\tfrac{\pi}{2}<\beta< \pi$\label{S0c}]{
\begin{tikzpicture}[line cap=round,line join=round,>=triangle 45,x=1.0cm,y=1.0cm, scale=1.9]
\fill[gray!30!white] (-0.3,-1.03) -- (-0.3,1.03) -- (0.98,1.03) -- (0.98,-1.03);
\fill[white] (-0.3,0) ellipse (0.3 and 0.583);
\fill[gray!30!white] (-0.98,0) -- (-0.98,1.03) -- (0.98,1.03) -- (0.98,0);
\draw[dotted] (-0.3,0) ellipse (0.3 and 0.583);
\draw [->] (-1,0) -- (1,0);
\draw (1,-0.15) node {$\real{s}$};
\draw [->] (0,-1) -- (0,1);
\draw (0.08,1.1) node {$\im{s}$};
\filldraw (-0.3,0.5) circle (1pt);
\draw (-0.6,0.55) node {$R\re^{\ri\beta}$};
\draw[decorate,decoration={coil,aspect=0,segment length=10pt,amplitude=2pt}] (-0.3,0.53) -- (-0.3,1.03);
\filldraw (-0.3,-0.5) circle (1pt);
\draw (-0.6,-0.45) node {$R\re^{-\ri\beta}$};
\draw[decorate,decoration={coil,aspect=0,segment length=10pt,amplitude=2pt}] (-0.3,-0.53) -- (-0.3,-1.03);
\draw (-0.3,-0.04) -- (-0.3,0.04) ;
\draw (-0.34,-0.14) node {$R\!\cos{\beta}$};
\draw (-0.04,0.583) -- (0.04,0.583) ;
\draw (0.18,0.583) node {$R$};
\draw (-0.04,-0.583) -- (0.04,-0.583) ;
\draw (0.23,-0.583) node {$-R$};
\end{tikzpicture} }
\caption{The cut complex $s$-plane $\cC_{R,\beta}$ and the region $\cR_{R,\beta}\subset \cC_{R,\beta}$ (shaded) in which $\arg{\mu(s)}\in[-\tfrac{\pi}{2},\pi]$. The boundary of the ellipse $\mathcal{E}_{R,\beta}$ described by equation \rf{EllipseEqn} is indicated by the dotted curve.}
\label{S0}
\end{center}
\end{figure} 
%
%

\begin{rem}
\label{OriginRem}
The region in which $f(s;R,\beta)$ is analytic and bounded, with a bound independent of $k$, can be extended to include a neighbourhood of the origin $s=0$. 
For $0< \beta < \pi$ let
$\eps_*=(R/2)\min\{1,1/(\sin\beta\sqrt{kR})\}$.
Then $f(s;R,\beta)$ is analytic in the ball $B_{\eps_*}(0)\subset\cC_{R,\beta}$. Moreover, if $s\in B_{\eps_*}(0)$ then \mbox{$R-\real{s}\cos{\beta}>R-\eps_*|\cos{\beta}|>R/2$}, and since $\real{r(s)}>0$ we can then estimate, for $s\in B_{\eps_*}(0)$,
\begin{align}
\label{}
|\mu(s)|=\frac{\sqrt{k}|s|\sin\beta}{\sqrt{\left|R-s\cos{\beta}+r(s)\right|}}&\leq \frac{\sqrt{k}|s|\sin\beta}{\sqrt{\left|\real{R-s\cos{\beta}+r(s)}\right|}} 
\leq \sqrt{\frac{2k}{R}}\eps_*\sin\beta
\leq \frac{1}{\sqrt{2}}.
\end{align}
Hence $\mu(s)$ is bounded in $B_{\eps_*}(0)$ with a bound that is independent of $k$, and this statement transfers to $f(s;R,\beta)$ because of the entirety of $F(\mu)$. 
However, although the bound on $|f(s;R,\beta)|$ implied by this result is independent of $k$, the region on which it holds varies with $k$ (through $\eps_*$); moreover shrinking as $k\to\infty$. In fact, no $k$-independent bound can hold on any $k$-independent neighbourhood of the origin. 
To see this, note that any such neighbourhood would include a point $s_\eps=\eps\,\re^{-\ri 3\pi/4}$ for some $0<\eps<1$ independent of $k$. It is easy to check that 
$\arg{\mu(s_\eps)}\in(-\tfrac{7\pi}{8},-\tfrac{9\pi}{16})$ and $|\mu(s_\eps)|=C\sqrt{kR}$ for some $C>0$ depending only on $\eps$ and $\beta$. Lemma~\ref{gLem} then implies that $|f(s_\eps;R,\beta)|$ tends to infinity exponentially fast as $k\to\infty$.
\end{rem}

We now turn to $h(s;R,\beta)=\pdonetext{\mu}{\bn}(r(s),\psi(s))$, which for $s>0$ can be written as 
\begin{align}
\label{eqn:dmudnDef}
h(s;R,\beta)
= \frac{k\sin\beta\,(r(s)-R)}{2r(s)\mu(s)},
\end{align}
and we adopt this formula as the analytic continuation of $h(s;R,\beta)$ to complex $s$. From \rf{eqn:dmudnDef} we see that $h(s;R,\beta)$ is analytic in the cut plane $\cC_{R,\beta}$, because the only singularities in $h(s;R,\beta)$ are at the branch points of $r(s)$, i.e\ at $s=R\re^{\pm\ri \beta}$. (The apparent singularity at $s=0$, where $\mu(0)=0$, is removable, since $r(0)=R$.) Moreover, one can show that, for any $0<\delta<1$,
\begin{align*}
\label{}
|h(s;R,\beta)|
\leq \frac{Ck}{\sqrt{kR\sin\beta}}, \qquad \left|\im{s}\right|\leq (1-\delta)R\sin\beta,
\end{align*}
where the constant $C>0$ depends only on $\delta$.

Combining these observations with the results in Lemma~\ref{hlem} allow us to prove the following theorem, which forms the basis of our $hp$ approximation results in \S\ref{sec:polygons}.
\begin{thm}
\label{cor:gReg}
For any $0<\delta<1$, the function $g(s;R,\beta)$ is analytic and bounded in 
\begin{align*}
\cS^\delta_{R,\beta}=\left\lbrace s\in\C:\, \left|\im{s}\right|<(1-\delta)R\sin\beta \textrm{ and } \left|\arg{s}\right|< \arctan\sqrt{(11+5\sqrt{5})/2}\right\rbrace;
\end{align*}
specifically, there exists a constant $C>0$ depending only on $\delta$ such that
\begin{align}
\label{}
|g(s;R,\beta)| \leq Ck\left(1+\frac{1}{\sqrt{kR\sin\beta}}\right), \qquad s\in\cS^\delta_{R,\beta}.
\end{align}
\end{thm}

\section{Numerical results}
\label{sec:NumResults}
The theoretical error estimates obtained in \S\ref{sec:HNAApproxSpace} have been validated by computing numerical approximations to the best approximation error 
\begin{align}
\label{eqn:BestApproxError}
\inf_{q\in \mathcal{P}^p_{\rm nc}} \left\|V - q\right\|_{L^2\left(\Gamma_{\rm nc}\right)}
\end{align}
appearing in the error estimate \rf{eqn:BestApproxtildev}. A sample of the results obtained is given in Figure~\ref{fig:NumResults}. Here the best approximation error \rf{eqn:BestApproxError} was computed by $L^2\left(\Gamma_{\rm nc}\right)$-orthogonal projection onto $\mathcal{P}^p_{\rm nc}$ of the exact solution $V$, the Fresnel integrals required for the evaluation of $V$ being computed using the algorithm of \cite{AlChLP:13}. The local basis for $\mathcal{P}^p_{\rm nc}$ on each element of $\mathcal{M}_{\rm nc}$ comprised appropriately scaled and shifted Legendre polynomials, and the integrals arising in the orthogonal projection were computed using high order Gaussian quadrature. In all of our experiments we took \mbox{$L_{\rm nc}=3/2$}, $L_{\rm nc}'=1$, $\sigma = 0.15$, and $c=1$. 

\begin{figure}
\begin{center}
\subfigure[$\alpha=\tfrac{3\pi}{4}$]{
\includegraphics[scale=0.59]{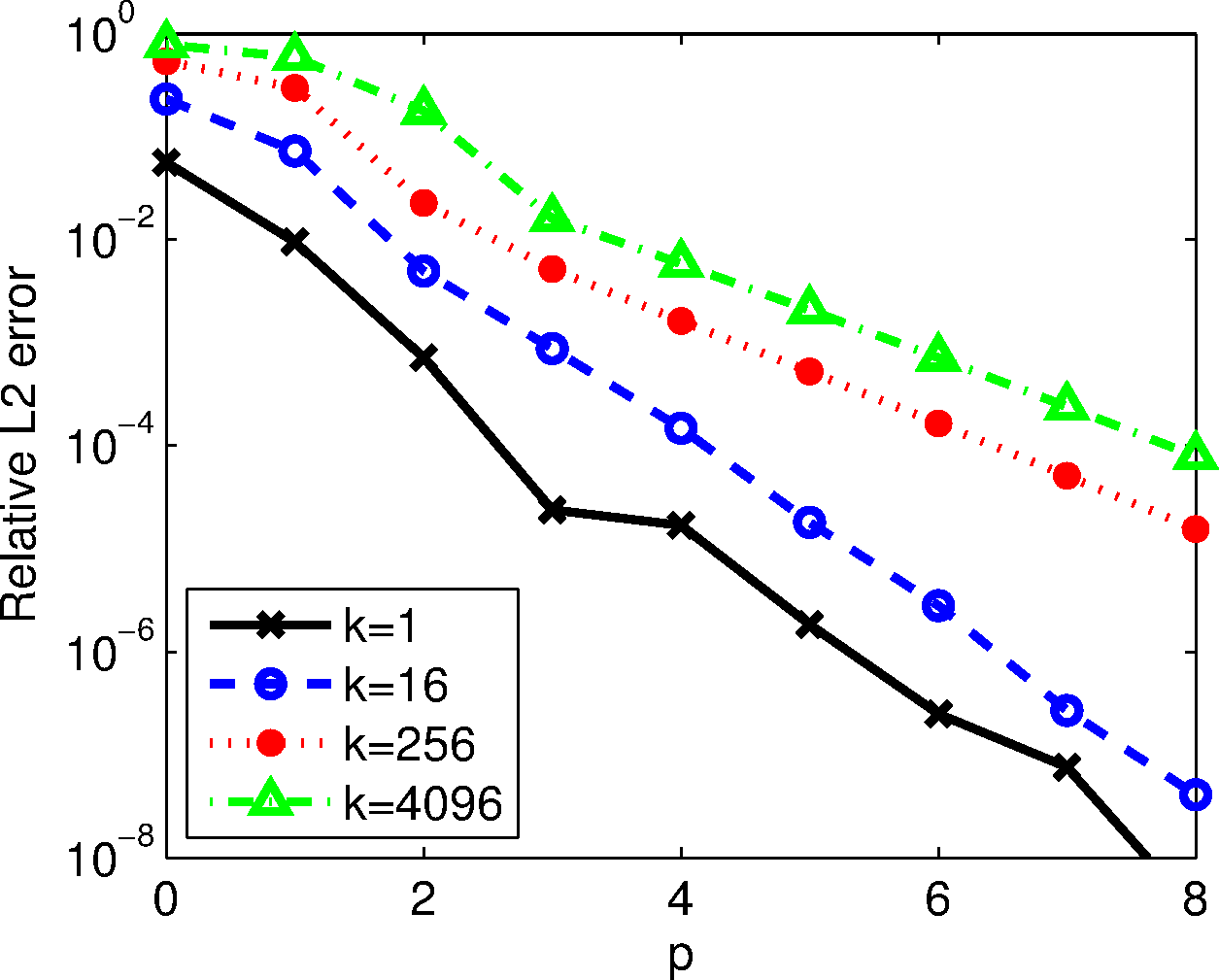}
}
\hs{2}
\subfigure[$k=16$]{
\includegraphics[scale=0.59]{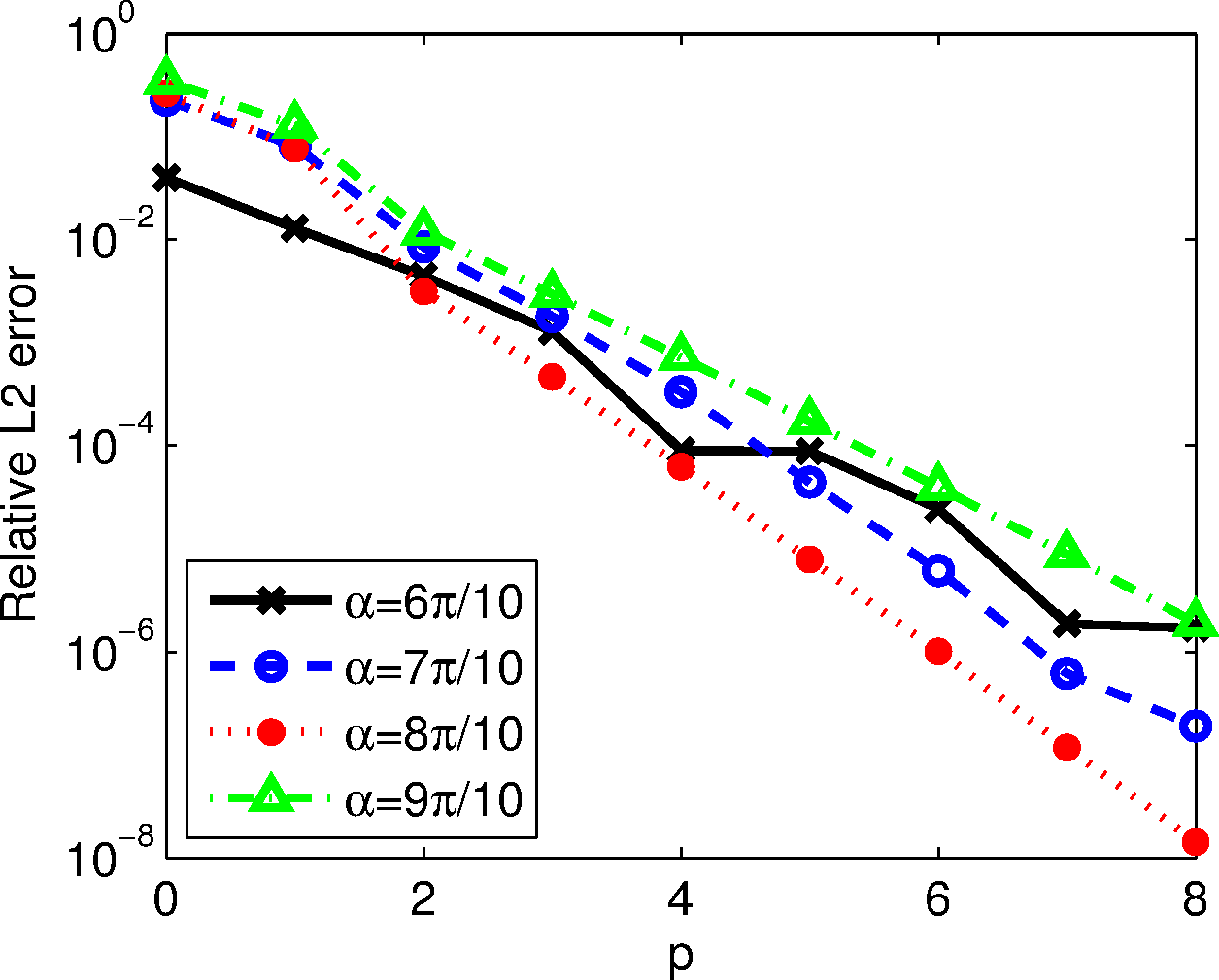}
}\\
\subfigure[$k=16$]{
\includegraphics[scale=0.59]{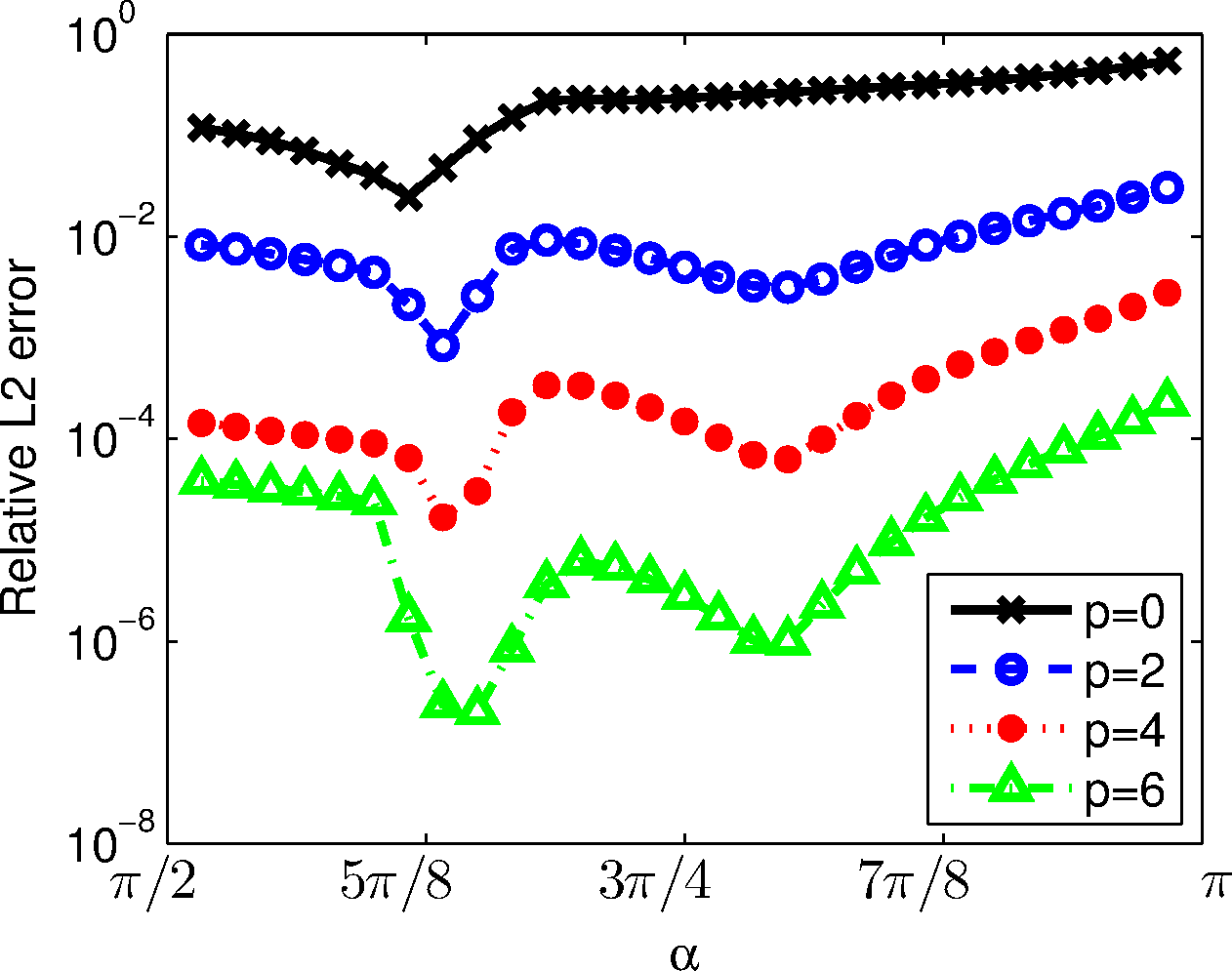}
}
\hs{2}
\subfigure[$k=16$]{
\includegraphics[scale=0.59]{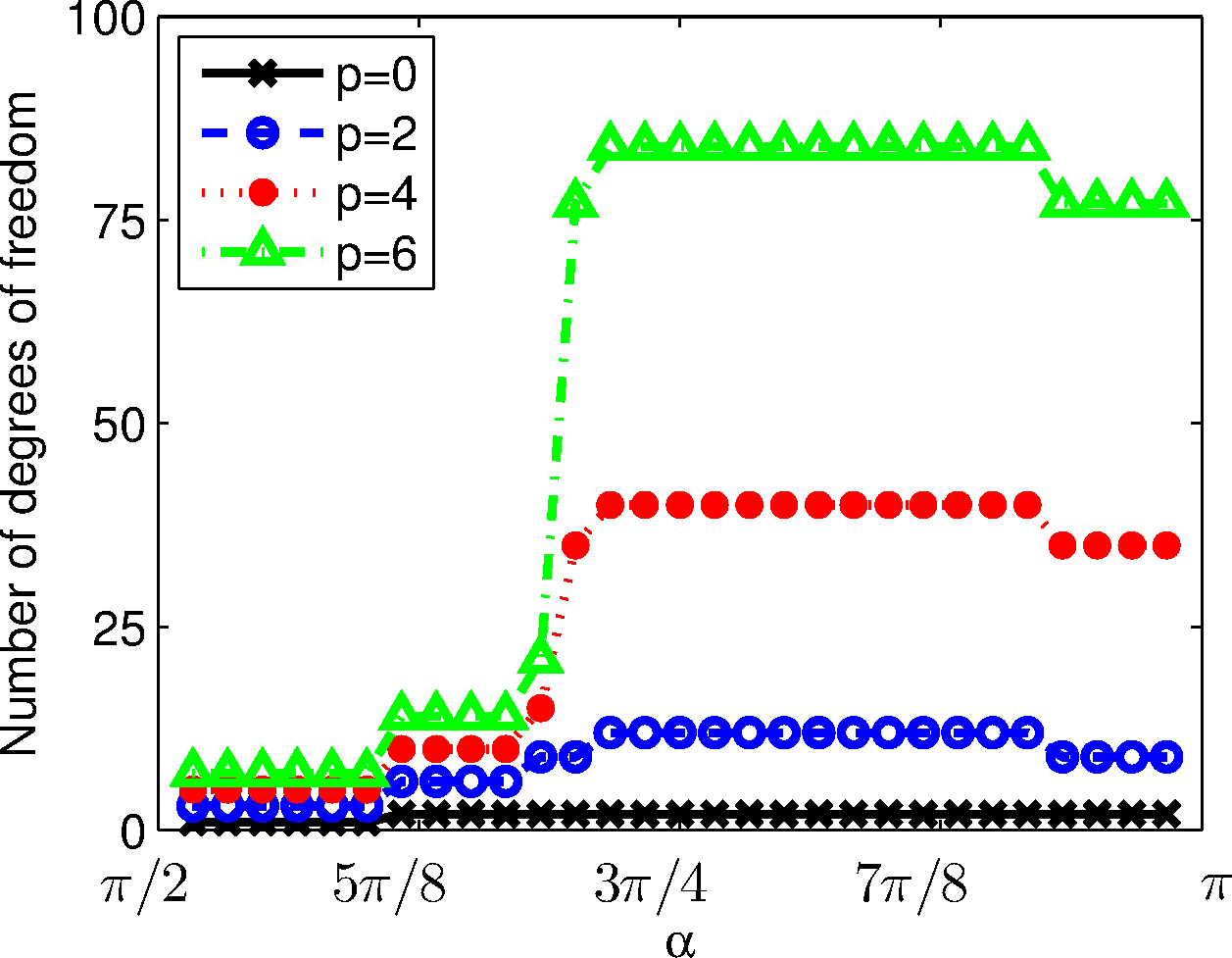}
}
\end{center}
\caption{Numerical results validating the best approximation error estimate \rf{eqn:BestApproxtildev}. Plots (a), (b) and (c) show how the relative error \rf{eqn:RelError} behaves as a function of $p$, $\alpha$ and $k$. Plot (d) shows the dependence of the number of degrees of freedom in $\mathcal{P}_{\rm nc}$ on $\alpha$. 
\label{fig:NumResults}
}
\end{figure}

Figure~\ref{fig:NumResults}(a) shows the relative error 
\begin{align}
\label{eqn:RelError}
\inf_{q\in \mathcal{P}^p_{\rm nc}} \left\|V - q\right\|_{L^2\left(\Gamma_{\rm nc}\right)}/\left\|V\right\|_{L^2\left(\Gamma_{\rm nc}\right)}
\end{align}
plotted against the polynomial degree $p$, for fixed incident angle $\alpha=\tfrac{3\pi}{4}$ and four different values of the wavenumber $k$. The exponential convergence predicted by \rf{eqn:BestApproxtildev} is clearly visible, and the error grows only relatively mildly as $k$ increases. Figure~\ref{fig:NumResults}(b) shows similar results for fixed $k$ and four different values of $\alpha$ between $\tfrac{\pi}{2}$ and $\pi$ (by symmetry it is sufficient to consider only $\alpha\in[\tfrac{\pi}{2},\pi]$). Figure~\ref{fig:NumResults}(c) investigates the dependence on $\alpha$ further, and confirms that, as predicted by our error estimate \rf{eqn:BestApproxtildev}, the best approximation converges uniformly in $\alpha$ as $p\to \infty$, with no blow-up in the error as $\alpha$ tends to $\tfrac{\pi}{2}$ or $\pi$, for example. Clearly the error is not independent of $\alpha$, but this is to be expected because the function $V$ and our approximation space $\mathcal{P}^p_{\rm nc}$ are both $\alpha$-dependent. Figure~\ref{fig:NumResults}(d) shows how the number of degrees of freedom in $\mathcal{P}_{\rm nc}$ depends on $\alpha$; one can clearly see the increase in the number of degrees of freedom as $\alpha$ increases through $\pi/2+\arctan{(L_{\rm nc}'/L_{\rm nc})}=\pi/2+\arctan{(2/3)}\approx (5.5)\times \tfrac{\pi}{8}$, the point at which the shadow boundary associated with the incident wave first intersects $\Gamma_{\rm nc}$. We also remark on the particularly small errors in Figure~\ref{fig:NumResults}(c) near $\alpha=\pi/2+\arctan{(L_{\rm nc}'/(3L_{\rm nc}/2))}=\pi/2+\arctan{(4/9)}\approx (5.1)\times \tfrac{\pi}{8}$, the point at which there is a mesh point exactly halfway along $\Gamma_{\rm nc}$ and the size of the largest element in $\mathcal{M}_{\rm nc}$ is at a minimum.
\section{Acknowledgements}
The author gratefully acknowledges support from EPSRC grant EP/F067798/1, and thanks Simon \mbox{Chandler-Wilde} and Stephen Langdon for helpful discussions in relation to this work.
\bibliography{BEMbib_short}
\bibliographystyle{siam}
\appendix
\section{Geometric meshes and polynomial approximation}
\label{app:meshes}
Given $-\infty<a<b<\infty$ and $p\in\N_0$, let $\mathcal{P}_p(a,b)$ denote the space of polynomials on $(a,b)$ of degree $\leq p$. A \emph{mesh} of $n\in\N$ elements on the interval $[a,b]$ is defined to be a set $\mathcal{M}=\{x_i\}_{i=0}^n$ such that $a=x_0<x_1<\ldots<x_n=b$. 
By the space of piecewise polynomials with degree $\leq p$ on the mesh $\mathcal{M}$ we mean the set
\begin{align*}
\label{}
\left\{P:[a,b]\to \C \,: \,P|_{(x_{i-1},x_i)} \in \mathcal{P}_{p}(x_{i-1},x_i),\, i=1,\ldots,n \right \}.
\end{align*}

In particular, given $L>0$ and $n\in\N$ we denote by $\mathcal{G}_n(0,L)=\{x_0,x_1,\ldots,x_n\}$ the geometric mesh on $[0,L]$ with $n$ layers, whose meshpoints $x_i$ are defined by
\begin{align*}
  x_0=0, \qquad x_i = \sigma^{n-i}L, \quad i=1,2,\ldots,n,
\end{align*}
where $0<\sigma<1$ is a fixed grading parameter.

Our best approximation estimates in this paper are based on the following standard result, which follows, e.g., from \cite[Theorem 2.1.1]{Stenger}.
\begin{lem}
\label{EllipseLem}
Let $a,b,d\in \R$ with $a<b$ and $d>b-a>0$. If the function $\phi$ is analytic and bounded in 
$\mathcal{E}_{a,b,d}=\left\{ s\in\mathbb{C}: |s-a|+|s-b|<d \right\}$ 
(the interior of the ellipse with foci $\{a,b\}$ and eccentricity $\epsilon=(b-a)/d<1$), then
\begin{align*}
\label{}
\inf_{q\in \mathcal{P}_p(a,b)}\norm{\phi-q}{L^\infty(a,b)}\leq\frac{2\rho^{-p}}{\rho-1}\norm{\phi}{L^\infty(\mathcal{E}_{a,b,d})}, 
\qquad \rho =1/\epsilon + \sqrt{1/\epsilon-1}>1.
\end{align*}
\end{lem}

\end{document}